\documentclass[a4paper,reqno]{amsart}

\usepackage{etoolbox}

\newtoggle{default}
\newtoggle{siam}
\newtoggle{birk}
\newtoggle{amsproc}
\newtoggle{wiley}

\toggletrue{default}
\usepackage[utf8]{inputenc}
\usepackage[english]{babel}
\usepackage{microtype} 

\iftoggle{siam}%
{\usepackage{amsfonts,amsopn,amsmath,amssymb}}
{\usepackage{amsmath,amssymb,amsthm}} 

\usepackage{thmtools,thm-restate} 
\usepackage{mathtools} 
\usepackage{upgreek}
\usepackage{orcidlink}

\usepackage{hyperref} 
\hypersetup{%
  colorlinks=true,
  linkcolor=blue,
  citecolor=green!60!black,
  pdftitle={Well-posedness and stability of the Lagrange representation of the n-D wave equation via boundary triples},
  pdfauthor={Bernhard Aigner; Nathanael Skrepek}
}
\usepackage{cleveref} 

\usepackage{tikz} 
\usetikzlibrary{cd}
\usetikzlibrary{calc}

\usepackage{pifont} 
\usepackage{nicefrac} 

\usepackage{enumitem}
\setlist[itemize]{labelindent=1em,leftmargin=*,itemsep=2pt,parsep=2pt}
\setlist[enumerate]{label=\textup{(\roman{*})},labelindent=3pt,leftmargin=*,listparindent=0pt,parsep=1pt,itemsep=0pt}

\definecolor{nscolor}{rgb}{0.0,0.0,0.8}

\definecolor{bacolor}{rgb}{0,0.5,0}

\makeatletter
\newcommand{\pushright}[1]{\ifmeasuring@#1\else\omit\hfill$\displaystyle#1$\fi\ignorespaces}
\makeatother

\newcommand{\argdot}{\boldsymbol{\cdot}}

\newcommand{\R}{\mathbb{R}}
\newcommand{\N}{\mathbb{N}} 

\newcommand{\C}{\mathbb{C}}

\newcommand{\Lp}[1]{\mathrm{L}^{#1}} 
\newcommand{\soboH}{\mathrm{H}} 
\newcommand{\cH}{\mathring{\soboH}}
\newcommand{\hH}{\hat{\soboH}}
\newcommand{\Lb}{\mathcal{L}_{\mathrm{b}}} 

\newcommand{\conC}{\mathrm{C}}
\newcommand{\Cc}[1][\infty]{\conC_{\mathrm{c}}^{#1}}

\newcommand{\iu}{\mathrm{i}} 
\newcommand{\euler}{\mathrm{e}} 
\newcommand{\e}{\euler} 
\newcommand{\dd}{\mathrm{d}} 
\newcommand{\dx}[1][x]{\mathop{\dd#1}}

\newcommand{\idop}{\mathrm{I}}

\newcommand{\cl}[2][]{\overline{#2}\ifthenelse{ \equal{#1}{} }{}{^{#1}}} 

\DeclareMathOperator{\ran}{ran}

\DeclareMathOperator{\supp}{supp}
\DeclareMathOperator{\dom}{dom}

\DeclareMathOperator{\Div}{div}
\renewcommand{\div}{\Div}

\newcommand{\grad}{\nabla}
\newcommand{\cgrad}{\mathring{\nabla}}

\DeclarePairedDelimiter{\set}{\{}{\}}
\DeclarePairedDelimiter{\norm}{\lVert}{\rVert}

\DeclarePairedDelimiterX{\dset}[2]{\{}{\}}{#1\,\delimsize\vert\,\mathopen{} #2}
\DeclarePairedDelimiterX{\scprod}[2]{\langle}{\rangle}{#1,#2}
\DeclarePairedDelimiterX{\dualprod}[2]{\langle}{\rangle}{#1,#2}

\renewcommand{\Re}{\operatorname{Re}}

\newcommand{\cpt}{\overset{\mathsf{cpt}}{\hookrightarrow}}

\newcommand{\boundtr}[1][]{\gamma_{0}\ifthenelse{\equal{#1}{}}{}{\big\vert_{#1}}}
\newcommand{\normaltr}[1][]{\gamma_{\nu}\ifthenelse{\equal{#1}{}}{}{\big\vert_{#1}}}

\newcommand{\adjun}{^{\ast}}

\newcommand{\hamiltonian}{\mathcal{Q}}
\newcommand{\nrg}{\mathrm{E}}



\ifboolexpr{togl{default} or togl{amsproc} or togl{birk}}{

\theoremstyle{plain}
\newtheorem{theorem}{Theorem}[section]
\newtheorem{lemma}[theorem]{Lemma}
\newtheorem{proposition}[theorem]{Proposition}
\newtheorem{corollary}[theorem]{Corollary}

\theoremstyle{definition}
\newtheorem{definition}[theorem]{Definition}

\theoremstyle{remark}
\newtheorem{remark}[theorem]{Remark}

}{}

\begin{document}

\title[Well-posedness and stability of the Lagrage representation]{Well-posedness and stability of the Lagrange representation of the n-D wave equation via boundary triples}




\begin{abstract}
  We study the Lagrange representation of the wave equation with generalized Laplacian $\div T \grad$.
  We allow the coefficients---the Young modulus $T$ and the density $\rho$---to be $\Lp{\infty}$ or even nonlocal operators.
  Moreover, the Lipschitz boundary of the domain $\Omega$ can be split into several parts admitting Dirichlet, Neumann and/or Robin-boundary conditions of displacement, velocity and stress.
  We show well-posedness of this classical model of the wave equation utilizing boundary triple theory for skew-adjoint operators. In addition, we show semi-uniform stability of solutions under slightly stronger assumptions by means of a spectral result.
\end{abstract}


\author[B.~Aigner]{Bernhard Aigner\,\orcidlink{0009-0009-8252-162X}}
\thanks{Supported by the state of Saxony via a graduate student stipend}

\address{Institute of Applied Analysis, TU Bergakademie Freiberg, Akademiestraße 6, 09599 Freiberg, Sachsen, Germany}
\email{bernhard.aigner@doktorand.tu-freiberg.de}


\author[N.~Skrepek]{Nathanael Skrepek\,\orcidlink{0000-0002-3096-4818}}

\address{Department of Applied Mathematics, University of Twente, P.O.\ Box 217, 7500 AE Enschede, Overijssel, The Netherlands}
\email{n.skrepek@utwente.nl}

\date{Dec 15, 2025}
\maketitle

\section{Introduction}%
\label{sec:introduction}


There is a plethora of studies on the wave equation and it is difficult to even quote the most significant ones. We simply name \cite{Pazy1983} for a classical semigroup approach, since we will employ a semigroup approach as well, and \cite{KurZwa2015}, since it regards the wave equation from the port-Hamiltonian perspective, but of course plenty of other tools are available as well, e.g., \cite{Waurick2022}.
%
The (classical) formulation of the wave equation on a bounded Lipschitz domain $\Omega \subseteq \R^{d}$ we are investigating in this article is the following second order partial differential equation
\begingroup%
\ifboolexpr{togl{default} or togl{birk} or togl{amsproc} or togl{siam}}{%
\thinmuskip=2mu%
\medmuskip=3mu%
\thickmuskip=3mu plus 1mu%
}{}%
\begin{align}
  \begin{aligned}
    \label{FullWE}
    \rho(\zeta) \tfrac{\partial^{2}}{\partial t^{2}} w(t,\zeta) &= \div T(\zeta) \grad w(t,\zeta) - a(\zeta) w(t,\zeta) - b(\zeta) \tfrac{\partial}{\partial t} w(t,\zeta), &&t \geq 0, \zeta \in \Omega, \\
    w(0,\zeta) &= w_{0}(\zeta), && \phantom{t \geq 0,\mathopen{}} \zeta \in \Omega, \\
    \tfrac{\partial}{\partial t} w(0,\zeta) &= w_{1}(\zeta), && \phantom{t \geq 0,\mathopen{}} \zeta \in \Omega.
  \end{aligned}
\end{align}
\endgroup
For the initial discussion we will exclude the terms $a$ and $b$, as we will take care of them with a perturbation argument later on. The coefficients $T$ and $\rho$ are the Young modulus and the material density, respectively. In the standard port-Hamiltonian approach, cf.\ \cite{KurZwa2015}, one introduces the new state variable $\begin{psmallmatrix} \grad w \\ \rho w_{t} \end{psmallmatrix}$, which yields the following representation of the wave equation
\begin{align*}
  \frac{\partial}{\partial t}
  \begin{pmatrix} \grad w \\ \rho w_{t}\end{pmatrix}
  =
  \underbrace{\begin{pmatrix}0 & \grad \\ \div & 0\end{pmatrix}}_{\eqcolon \mathrlap{\mathcal{J}_{1}}}
  \underbrace{\begin{pmatrix} T & 0 \\ 0 & \frac{1}{\rho} \end{pmatrix}}_{\eqcolon \mathrlap{\hamiltonian_{1}}}
  \begin{pmatrix} \grad w \\ \rho w_{t} \end{pmatrix}.
\end{align*}
This is the so-called \emph{Dirac representation}, which is comprised of the formally skew-adjoint $\mathcal{J}_{1}$ (representing the underlying Dirac structure/subspace) and the bounded, positive and self-adjoint $\hamiltonian_{1}$ (representing the Lagrangian structure/subspace).
However, in \cite{BeHaLeMa2024} an alternative so-called \emph{Lagrange representation} was proposed, which uses
$\begin{psmallmatrix} w \\ \rho w_{t}\end{psmallmatrix}$ as state variable. The corresponding system is
\begin{equation}\label{eq:Lagrange-abstract-Cauchy-problem}
  \frac{\partial}{\partial t}
  \begin{pmatrix} w \\ \rho w_{t} \end{pmatrix}
  =
  \underbrace{\begin{pmatrix}0 & \idop \\ -\idop & 0\end{pmatrix}}_{\eqcolon\mathrlap{\mathcal{J}_{2}}}
  \underbrace{\begin{pmatrix} - \div T \grad & 0 \\ 0 & \frac{1}{\rho} \end{pmatrix}}_{\eqcolon\mathrlap{\hamiltonian_{2}}}
  \begin{pmatrix} w \\ \rho w_{t}\end{pmatrix}.
\end{equation}
Similar to the Dirac representation, $\mathcal{J}_{2}$ is skew-adjoint and $\hamiltonian_{2}$ is formally self-adjoint and positive. In contrast to the Dirac-representation however, $\mathcal{J}_{2}$ is bounded and $\hamiltonian_{2}$ is unbounded. This formulation also corresponds to the standard method to turn the wave equation into a first order problem, cf.\ \cite[Sec.~7.4]{Pazy1983}.

Naturally, employing terms like (un)bounded and self-adjoint warrants specifying the state spaces. For the Dirac representation one chooses $\Lp{2}(\Omega;\C^{d}) \times \Lp{2}(\Omega;\C)$ and for the Lagrange representation it comes naturally to choose $\soboH^{1}(\Omega;\C) \times \Lp{2}(\Omega;\C)$. In order to analyze well-posedness for the Dirac representation, it suffices to analyze $\mathcal{J}_{1}$, as $\hamiltonian_{1}$ can be incorporated into the inner product and thus does not play a role, cf.\ \cite[Lem.~7.2.3]{JacZwa2012}. We want to mimic this approach for the Lagrange representation, but since $\hamiltonian_{2}$ is unbounded, it does not immediately provide an equivalent inner product on the state space, i.e., we want to regard the following almost inner product
\begin{multline*}
  \scprod{x}{y}_{\hamiltonian_{2}}
  \coloneq
  \scprod{\hamiltonian_{2} x}{y}
  = \scprod*{%
    \begin{pmatrix} -\div T \grad & 0 \\ 0 & \frac{1}{\rho}\end{pmatrix}
    \begin{pmatrix}x_{1} \\ x_{2}\end{pmatrix}
  }{\begin{pmatrix}y_{1} \\ y_{2}\end{pmatrix}} \\
  = \scprod{\tfrac{1}{\rho} x_{2}}{y_{2}} - \scprod{\div T \grad x_{1}}{y_{1}}\text{.}
\end{multline*}
Applying integration by parts yields
\begin{equation*}
  \scprod{x}{y}_{\hamiltonian_{2}}
  = \scprod{\tfrac{1}{\rho} x_{2}}{y_{2}} + \scprod{T \grad x_{1}}{\grad y_{1}} - \scprod*{\nu \cdot T \grad x_{1}\big\vert_{\partial\Omega}}{y_{1}\big\vert_{\partial\Omega}}\text{,}
\end{equation*}
where $\nu$ denotes the unit normal vector on $\partial\Omega$. It is worth pointing out that this expression is only an inner product on the state space $\soboH^{1}(\Omega;\C) \times \Lp{2}(\Omega;\C)$ if we additionally impose a boundary condition such as $\nu \cdot T \grad x_{1}\big\vert_{\partial\Omega} = - k_{1} x_{1}\big\vert_{\partial\Omega}$, where $k_{1}$ is a positive semi-definite operator\footnote{We are in particular interested in multiplication operators that may vanish on parts of the boundary.} on $\Lp{2}(\partial\Omega;\C)$. In that case we obtain
\begin{equation*}
  \scprod{x}{y}_{\hamiltonian_{2}}
  = \scprod{\tfrac{1}{\rho} x_{2}}{y_{2}} + \scprod{T \grad x_{1}}{\grad y_{1}} + \scprod*{k_{1} x_{1}\big\vert_{\partial\Omega}}{y_{1}\big\vert_{\partial\Omega}}
\end{equation*}
and we can appeal to the Friedrichs/Poincar\'e inequality. Alternatively, the boundary condition $x_{1} \big\vert_{\partial\Omega} = 0 = y_{1}\big\vert_{\partial\Omega}$ would eliminate the boundary parts completely. In either case, the above expression induces an equivalent inner product on the state space $\soboH^{1}(\Omega;\C) \times \Lp{2}(\Omega;\C)$.
Note that the energy of a state $x = \begin{psmallmatrix} x_{1} \\ x_{2} \end{psmallmatrix} \in \soboH^{1}(\Omega;\C) \times \Lp{2}(\Omega;\C)$ is given by
\begin{equation*}
  \nrg(x) \coloneq \scprod{\tfrac{1}{\rho} x_{2}}{x_{2}} + \scprod{T \grad x_{1}}{\grad x_{1}}.
\end{equation*}
Therefore the inner product $\scprod{\argdot}{\argdot}_{\hamiltonian_{2}}$ is composed of an energy part and a boundary part.
We will construct a boundary triple for the operator in the abstract Cauchy equation~\eqref{eq:Lagrange-abstract-Cauchy-problem}, which will enable us to parameterize all dissipative boundary conditions, i.e., boundary conditions where the solution does not grow. Loosely speaking, we will come to the conclusion that boundary conditions of the form
\begin{equation*}
  k_{1} x_{1} \big\vert_{\partial\Omega} + \nu \cdot T \grad x_{1} \big\vert_{\partial\Omega} + k_{2} \tfrac{1}{\rho} x_{2}\big\vert_{\partial\Omega} = 0,
\end{equation*}
or in terms of $w$
\begin{equation*}
   k_{1} w \big\vert_{\partial\Omega} + \nu \cdot T \grad w \big\vert_{\partial\Omega} + k_{2} \partial_{t}w \big\vert_{\partial\Omega}= 0,
\end{equation*}
where $k_{1}$ is the positive semi-definite operator from before and $k_{2}$ is another positive semi-definite operator on $\Lp{2}(\partial\Omega;\C)$, will be well-posed. We point out that in contrast to the Dirac representation, cf.\ \cite{KurZwa2015,Arendt2023}, we can formulate boundary conditions that involve the displacement $w$ itself (in addition to the velocity $\partial_{t} w$ and the normal stress $\nu \cdot T \grad w$).


Additionally, since $k_{1}$ and $k_{2}$ are allowed to be semi-definite, we can split the boundary into five (possibly empty) parts and allow the following boundary conditions:
\begin{align}
  \begin{alignedat}{3}
    \label{eq:boundary-conditions}
    w &= 0 & \quad& \text{on}\ \Gamma_{0}, \\
    \nu \cdot T \grad w &= 0 && \text{on}\ \Gamma_{1}, \\
    k_{1} w + \nu \cdot T \grad w &= 0 && \text{on}\ \Gamma_{2}, \\
    \nu \cdot T \grad w + k_{2} \partial_{t}w &= 0 && \text{on}\ \Gamma_{3}, \\
    k_{1}w + \nu \cdot T \grad w + k_{2}\partial_{t}w& = 0 && \text{on}\ \Gamma_{4},
  \end{alignedat}
\end{align}
where we assume $k_{1}$ to be nonzero almost everywhere on $\Gamma_{2}\cup \Gamma_{4}$ and $k_{2}$ to be nonzero almost everywhere on $\Gamma_{3}\cup \Gamma_{4}$, i.e., the vanishing parts are being accounted for by $\Gamma_{0},\Gamma_{1}$ and $\Gamma_{3}$ or by $\Gamma_{0},\Gamma_{1}$ and $\Gamma_{2}$ respectively.

Note that $\Gamma_{0} = \emptyset$ or $\Gamma_{2} \cup \Gamma_{4} = \emptyset$ ($k_{1} = 0$) is allowed, but not simultaneously, i.e., we require
\begin{equation*}
  \Gamma_{0} \neq \emptyset \quad\text{or}\quad k_{1} \neq 0.
\end{equation*}
This is necessary to make sure that $\scprod{\argdot}{\argdot}_{\hamiltonian_{2}}$ is an inner product.

In this work, boundary triples are used to show the well-posedness of the boundary conditions \eqref{eq:boundary-conditions}. In a related context, boundary triples have also proven useful in \cite{Gernandt2025}, where they are employed to investigate boundary conditions ensuring a Lagrangian subspace.

Following up on well-posedness, we will investigate stability of solutions. The dissipative relation on the boundary can be viewed as a boundary feedback or damping. Damped wave equations have been studied by multiple authors, in particular Zuazua (e.g., \cite{Zuazua1990}), but more recently (even for the delay case) by Pignotti et.\ al.\ (e.g., \cite{Pignotti2010,Pignotti2011}). We point out, that opposed to our approach in this article, stability results for the wave equation usually cover the case, where the damping happens in the interior and geometric conditions (such as the Geometric Control Condition) are a necessary cost to pay.

We will pursue an alternative strategy, following \cite{Jacob2021} in our approach, settling for a weaker notion than exponential stability, by the name of ``semi-uniform stability'', cf.\ \Cref{sec:stability}. Slightly stronger assumptions on the coefficients $T$ and $\rho$ in conjunction with dissipativity will prove enough to show at least semi-uniform stability of solutions to problem \eqref{FullWE}. In particular, we have to show that there is no spectrum on the imaginary axis. A similar approach has been applied to Maxwell's equations in \cite{SkrepekWaurick2024}.

\subsection*{Assumptions}
To turn this outline into rigorous mathematics, we now state the basic assumptions that will be used throughout the remainder of the article, unless explicitly stated otherwise.
\begin{enumerate}[label=\textup{(A\arabic{*})}]
  \item Let $\Omega$ be a bounded and connected Lipschitz domain in $\R^{d}$.

  \item\label{item:assumption-Young-modulus} Let $T \in \Lb(\Lp{2}(\Omega;\C^{d}))$ with $c^{-1}\idop < T < c\idop$ for some $c>0$.\footnote{$\Lb(X)$ denotes the linear and bounded operators on a normed vector space $X$.}
        \\
        For stability additionally: $T \in \Lp{\infty}(\Omega;\C^{d\times d})$ and Lipschitz continuous.\footnote{We identify the $\Lp{\infty}$ function with the induced multiplication operator.}

  \item Let $\rho \in \Lb(\Lp{2}(\Omega;\C))$ with $c^{-1}\idop <\rho < c\idop$ for some $c>0$.
        \\
        For stability additionally: $\rho \in \Lp{\infty}(\Omega;\C)$ and Lipschitz continuous.\footnotemark[\value{footnote}]
  
  \item The boundary is split into five open, disjoint and possible empty parts $\Gamma_{0}, \Gamma_{1}, \Gamma_{2}, \Gamma_{3}, \Gamma_{4} \subseteq \partial\Omega$ that satisfy
        \begin{equation*}
          \cl{\Gamma_{0} \cup \Gamma_{1} \cup \Gamma_{2} \cup \Gamma_{3} \cup \Gamma_{4}} = \partial\Omega
          \quad\text{and}\quad
          \sum_{i=0}^{4}\mu(\Gamma_{i}) = \mu(\partial\Omega),
        \end{equation*}
        where $\mu$ is the surface measure of $\partial\Omega$. These conditions ensure that the $\Gamma$'s cover $\partial\Omega$ in the sense that the uncovered parts are negligible.

  \item Let $w_{0}\in \cH^{1}_{\Gamma_{0}}(\Omega;\C)$ and $w_{1}\in \Lp{2}(\Omega;\C)$.\footnote{Here $\cH^{1}_{\Gamma_{0}}(\Omega) = \dset{v\in \soboH^{1}(\Omega)}{v\vert_{\Gamma_{0}}=0}$. We will come back to this space in \Cref{sec:well-posedness}.}
        
  \item\label{item:assumption-boundary-operators}
        We assume $k_{1},k_{2} \in \Lp{\infty}(\partial\Omega \setminus \cl{\Gamma_{0}};\R)$ such that $k_{1},k_{2} \geq 0$ and $k_{1} \neq 0$ or $\Gamma_{0} \neq \emptyset$.
        \\
        For stability additionally: $\exists \Gamma \subseteq \partial\Omega \setminus \cl{\Gamma_{0}}$ non-empty open such that $\supp k_{2} \supseteq \Gamma$.\footnote{This just means that the damping acts on an open set.}

  \item Let $a,b \in \Lp{\infty}(\Omega;\C)$.\footnote{More general assumptions are possible, cf.\ \cref{sec:well-posedness,sec:stability}.}
\end{enumerate}

\begin{remark}
  Assumption~\ref{item:assumption-boundary-operators} is not as general as possible, because otherwise the formulation would be a bit cumbersome. We could relax the conditions to: $k_{1}, k_{2} \in \Lb(\Lp{2}(\partial\Omega \setminus \cl{\Gamma_{0}}))$ positive semi-definite such that
  $1 \notin \ker k_{1}$ or $\Gamma_{0} \neq \emptyset$. For stability we additionally ask for
  \begin{equation}\label{eq:vanish-on-open-set}
    k_{2}f = 0 \quad\implies\quad \exists \Gamma \subseteq \partial\Omega \setminus \cl{\Gamma_{0}} \ \text{open and non-empty}\, : \, f \big\vert_{\Gamma} = 0.
  \end{equation}
  The condition $1 \notin \ker k_{1}$ is necessary for Friedrichs/Poincar{\'e} inequality and \eqref{eq:vanish-on-open-set} is necessary to apply the unique continuation principle.
\end{remark}

%
%
%

In the following, we first present a short preliminary section, which serves as a reminder for well-established concepts such as trace operators and boundary triples, which we will employ to prove our main results concerning well-posedness (\Cref{sec:well-posedness}) and stability (\Cref{sec:stability}).
We will end our article with a conclusion and a short appendix on the applicability of a unique continuation theorem in \Cref{sec:stability} and an additional result regarding regularity in the case of the trace maps taking values in $\Lp{2}(\partial\Omega; \C)$.

\section{Preliminaries}
\label{sec:preliminaries}

This section is largely a short recapitulation of established theory that we present both for the convenience of the reader, as well as to establish notation that we are going to use in consecutive sections.

Note that we will often not specify the codomain of $\Lp{p}$ spaces, Sobolev spaces and spaces of continuous functions as it will be clear from the context, e.g., we write simply $\Lp{2}(\Omega)$ for both $\Lp{2}(\Omega;\C)$ and $\Lp{2}(\Omega;\C^{k})$.

\subsection{Sobolev Spaces}
We clarify/introduce some notation first: We let $\soboH^{1}(\Omega)$ be the space of all $\Lp{2}$-functions with distributional derivative in $\Lp{2}(\Omega)$ also, together with the norm $\norm{\argdot}_{\soboH^{1}(\Omega)}\coloneq \sqrt{\norm{\argdot}_{\Lp{2}(\Omega)}^{2}+\norm{\grad\argdot}_{\Lp{2}(\Omega)}^{2}}$. In other words, $\soboH^{1}(\Omega) = \bigl(\dom (\grad), \norm{\argdot}_{\dom(\grad)}\bigr)$, where $\grad$ is the weak gradient on $\Lp{2}(\Omega)$.
Similarly, the corresponding space for the divergence operator $\div f \coloneq \sum_{i=1}^{d} \partial_{i} f_{i}$ is
\begin{equation*}
  \soboH(\div,\Omega) \coloneq \dset[\big]{f \in \Lp{2}(\Omega;\C^{d})}{\div f \in \Lp{2}(\Omega;\C) \ \text{(in the distributional sense)}}.
\end{equation*}
Note that $\Cc(\cl{\Omega};\C^{d}) \coloneq \dset*{f \big\vert_{\cl{\Omega}}}{f \in \Cc(\R^{d};\C^{d})}$ is dense in $\soboH(\div,\Omega)$, see, e.g.,  \cite[Ch.~IX Part A Sec.~2 Thm.~1]{DautrayLions2000} or \cite[Thm.~3.18]{Sk21}.

\subsection{Trace Operators}
The following can be found in more detail in \cite[Appendix~A]{KurZwa2015}.
For $f\in \Cc(\cl{\Omega})$ we can define the map
\begin{align*}
  \tilde{\gamma}_{0}\colon
  \left\{
  \begin{array}{rcl}
    \Cc(\cl{\Omega}) & \to & \Lp{2}(\partial\Omega), \\
    f & \mapsto & f \big\vert_{\partial\Omega}.
  \end{array}
  \right.
\end{align*}
We can extend $\tilde{\gamma}_{0}$ to a continuous map $\boundtr \colon \soboH^{1}(\Omega) \to \Lp{2}(\partial\Omega)$, the \emph{Dirichlet trace operator}. We define its image as
\begin{equation*}
  \soboH^{\frac{1}{2}}(\partial\Omega) \coloneq \ran \boundtr ,
  \quad\text{equipped with }\quad
  \norm{\phi}_{\soboH^{\frac{1}{2}}(\partial\Omega)} \coloneq \inf \dset{\norm{g}_{\soboH^{1}(\Omega)}}{\boundtr g = \phi}.
\end{equation*}
This space is even a Hilbert space, since it can be represented by the quotient space $\soboH^{1}(\Omega) / \ker \boundtr$.
Note that $\boundtr$ is a continuous map from $\soboH^{1}(\Omega)$ to $\soboH^{\frac{1}{2}}(\partial\Omega)$ and $\soboH^{\frac{1}{2}}(\partial\Omega)$ is continuously embedded in $\Lp{2}(\partial\Omega)$.
Let $\Gamma_{0}, \tilde{\Gamma} \subseteq \partial\Omega$ be a (measurable) partition of $\partial\Omega$ (up to sets of measure zero).
We can restrict $\boundtr f$ to $\Gamma_{0}$, which makes $\boundtr \big\vert_{\Gamma_{0}} \colon f \mapsto \boundtr f \big\vert_{\Gamma_{0}}$ continuous from $\soboH^{1}(\Omega)$ to $\Lp{2}(\Gamma_{0})$ also. Hence,
\begin{equation*}
  \cH^{1}_{\Gamma_{0}}(\Omega) \coloneq \ker \boundtr \big\vert_{\Gamma_{0}} = \dset*{f \in \soboH^{1}(\Omega)}{\boundtr f \big\vert_{\Gamma_{0}} = 0}
\end{equation*}
is closed and therefore a Hilbert space when equipped with $\norm{\argdot}_{\soboH^{1}(\Omega)}$. We define the corresponding trace space\footnote{Note that $\cH^{\frac{1}{2}}(\tilde{\Gamma})$ is not the same as the fractional Sobolev space $\soboH^{\frac{1}{2}}(\tilde{\Gamma})$, as not every function in $\soboH^{\frac{1}{2}}(\tilde{\Gamma})$ can be extended by $0$ to elements of $\soboH^{\frac{1}{2}}(\partial\Omega)$, e.g., constant functions.}
\begin{equation*}
  \cH^{\frac{1}{2}}(\tilde{\Gamma}) \coloneq \boundtr \cH^{1}_{\Gamma_{0}}(\partial\Omega)
\end{equation*}
and endow it with $\norm{\argdot}_{\soboH^{\frac{1}{2}}(\partial\Omega)}$. As the functions in $\cH^{\frac{1}{2}}(\tilde{\Gamma})$ are $0$ on $\Gamma_{0}$, we usually regard this space as continuously embedded in $\Lp{2}(\tilde{\Gamma})$.
We define the dual space of $\cH^{\frac{1}{2}}(\tilde{\Gamma})$ with pivot space $\Lp{2}(\tilde{\Gamma})$ as
\begin{equation*}
  \soboH^{-\frac{1}{2}}(\tilde{\Gamma}) \coloneq \bigl(\cH^{\frac{1}{2}}(\tilde{\Gamma})\bigr)^{\prime},
\end{equation*}
i.e., $\cH^{\frac{1}{2}}(\tilde{\Gamma}) \subseteq \Lp{2}(\tilde{\Gamma}) \subseteq \soboH^{-\frac{1}{2}}(\tilde{\Gamma})$ forms a Gelfand triple. The pairing between $\cH^{\frac{1}{2}}(\tilde{\Gamma})$ and $\soboH^{-\frac{1}{2}}(\tilde{\Gamma})$ is an extension of the inner product of $\Lp{2}(\tilde{\Gamma})$:
\begin{equation*}
  \dualprod{f}{g}_{\soboH^{-\frac{1}{2}}(\tilde{\Gamma}),\cH^{\frac{1}{2}}(\tilde{\Gamma})} \coloneq
  \lim_{n\to \infty}\scprod{f_{n}}{g}_{\Lp{2}(\tilde{\Gamma})},
\end{equation*}
where $(f_{n})_{n\in\N}$ is a sequence in $\Lp{2}(\tilde{\Gamma})$ converging to $f \in \soboH^{-\frac{1}{2}}(\tilde{\Gamma})$ (w.r.t.\ $\norm{\argdot}_{\soboH^{-\frac{1}{2}}(\tilde{\Gamma})}$) and $g \in \cH^{\frac{1}{2}}(\tilde{\Gamma})$. Moreover, we define $\dualprod{g}{f}_{\cH^{\frac{1}{2}}(\tilde{\Gamma}),\soboH^{-\frac{1}{2}}(\tilde{\Gamma})}$ as the complex conjugate of $\dualprod{f}{g}_{\soboH^{-\frac{1}{2}}(\tilde{\Gamma}),\cH^{\frac{1}{2}}(\tilde{\Gamma})}$ and we will use the short notation
\begin{align*}
  \dualprod{f}{g}_{\mp \frac{1}{2}} \coloneq  \dualprod{f}{g}_{\soboH^{-\frac{1}{2}}(\tilde{\Gamma}),\cH^{\frac{1}{2}}(\tilde{\Gamma})}
  \quad\text{and}\quad
  \dualprod{g}{f}_{\pm \frac{1}{2}} \coloneq \dualprod{g}{f}_{\cH^{\frac{1}{2}}(\tilde{\Gamma}),\soboH^{-\frac{1}{2}}(\tilde{\Gamma})}.
\end{align*}


For $f\in \Cc(\cl{\Omega};\C^{d})$ we can define the map
\begin{align*}
  \tilde{\gamma}_{\nu}\colon
  \left\{
  \begin{array}{rcl}
   \Cc(\cl{\Omega};\C^{d}) & \to & \Lp{2}(\tilde{\Gamma}), \\
    f & \mapsto & \nu \cdot f \big\vert_{\tilde{\Gamma}}.
  \end{array}
  \right.
\end{align*}
Via continuous extension we can extend $\tilde{\gamma}_{\nu}$ to $\normaltr\colon \soboH(\div,\Omega)\to \soboH^{- \frac{1}{2}}(\tilde{\Gamma})$, the \emph{normal trace operator}.

\begin{theorem}[{\cite[Theorem~A.8]{KurZwa2015}}] 
  \label{th:NormalTrace}
  The \emph{normal trace operator} $\normaltr \colon \soboH(\div,\Omega) \to \soboH^{- \frac{1}{2}}(\tilde{\Gamma})$ is a bounded, linear and surjective operator.
\end{theorem}

Both trace operators  together give rise to the following integration by parts rule:
\begin{theorem}[Integration by parts]
  \label{th:PI}
  For any $F\in \soboH(\div,\Omega)$ and $g\in \cH_{\Gamma_{0}}^{1}(\Omega)$ there holds:
  \begin{equation*}
    \scprod{\div F}{g}_{\Lp{2}(\Omega)}
    + \scprod{F}{\grad g}_{\Lp{2}(\Omega)}
    = \dualprod{\normaltr F}{\boundtr g}_{\soboH^{-\frac{1}{2}}(\tilde{\Gamma}),\cH^{\frac{1}{2}}(\tilde{\Gamma})}.
  \end{equation*}
\end{theorem}

\begin{proof}
  For $F \in \conC^{\infty}(\cl{\Omega})^{d}$ and $g\in \cH_{\Gamma_{0}}^{1}(\Omega)$ we have:
  \begin{equation*}
    \int_{\Omega}\scprod{\div F}{g}_{\C} \dx[\uplambda]
    + \int_{\Omega} \scprod{F}{\grad g}_{\C^{d}} \dx[\uplambda]
    = \int_{\tilde{\Gamma}}(\nu \cdot F) \overline{g} \dx[\mu].
  \end{equation*}
  Density of $\Cc(\cl{\Omega};\C^{d})$ in $\soboH(\div,\Omega)$ implies the claim.
\end{proof}



The Friedrichs/Poincar\'e inequality (\Cref{th:Poincare}) allows $\soboH^{1}(\Omega)$ to be equipped with an equivalent inner product:
\begin{equation}\label{eq:equivalent-H1-norm}
  \scprod{f}{g}_{\soboH^{1}(\Omega)}\coloneq \scprod{\grad f}{\grad g}_{\Lp{2}(\Omega)}+\scprod{k_{1}\boundtr f}{\boundtr g}_{\Lp{2}(\tilde{\Gamma})}.
\end{equation}
In particular we are interested in $\cH^{1}_{\Gamma_{0}}(\Omega)$, which is a subset of $\soboH^{1}(\Omega)$.
Note that \eqref{eq:equivalent-H1-norm} still defines an equivalent norm on $\cH^{1}_{\Gamma_{0}}(\Omega)$ if $k_{1} = 0$, provided that $\Gamma_{0} \neq \emptyset$. Hence, we need
\begin{equation*}
  k_{1} \neq 0 \quad\text{or}\quad \Gamma_{0} \neq \emptyset.
\end{equation*}
\subsection{Dissipative operators}

We remind the reader of a few commonly known facts about dissipative operators:
\begin{definition}
  A linear operator $A\colon H\supseteq \dom A\to H$ on a Hilbert space $H$ is called \emph{dissipative}, if $\Re \scprod{Ax}{x}_{H} \leq 0$ for all $x\in \dom A$. $A$ is called \emph{maximally dissipative} if there is no proper dissipative extension of $A$.
\end{definition}



Here are a few characterizations of maximally dissipative operators:
\begin{proposition}
  Let $A$ be a densely defined, linear operator on a Hilbert space $H$. The following are equivalent:
  \begin{enumerate}[leftmargin=5ex]
    \item $A$ is maximally dissipative.
    \item $A$ is dissipative and $\ran (\lambda I - A)$ is onto for some (and hence all) $\lambda >0$.
    \item 
          $A$ is closed and both $A$ and $A^{\ast}$ are dissipative.
  \end{enumerate}
\end{proposition}

Proofs of these statements as well as the next theorem are contained in \cite[Sec.~II.3b]{Engel2006}. We will rely on the following well-known theorem:

\begin{theorem}[Lumer--Philips]
  \label{th:Lumer-Philips}
  Let $A\colon H\supseteq \dom A\to H$ be a densely defined, linear operator on a Hilbert space $H$. Then $A$ is infinitesimal generator of a contraction semigroup if and only if $A$ is maximally dissipative.
\end{theorem}

\subsection{Helmholtz decomposition}
The classical Helmholtz decomposition allows the splitting of a vector field into a gradient field and a divergence-free vector field. We will use a slight modification of the classical Helmholtz decomposition and introduce
\begin{align*}
  \cH^{1}(\Omega) &\coloneq \cH_{\partial\Omega}^{1}(\Omega) = \dset{f \in \soboH^{1}(\Omega)}{\boundtr f = 0}
  \\
  \soboH(\div 0,\Omega) &\coloneq \dset{v\in \soboH(\div,\Omega)}{\div v=0} = \ker \div.
\end{align*}

With this notation, we show the following modification of the classical result:
\begin{theorem}[Helmholtz decomposition]
  \label{th:Helmholtz}
  Let
  $T\in \Lb(\Lp{2}(\Omega))$ be as in \ref{item:assumption-Young-modulus}, i.e., $T$ positive definite such that $c^{-1}\idop < T < c\idop$ for some $c>0$. Then
  \begin{equation*}
    \Lp{2}(\Omega)= T \grad \cH^{1}(\Omega)\oplus_{T^{-1}} \soboH(\div 0,\Omega),
  \end{equation*}
  where $\oplus_{T^{-1}}$ denotes the orthogonal sum w.r.t.\ the inner product $\scprod{f}{g}_{T^{-1}} \coloneq \scprod{T^{-1}f}{g}$.
\end{theorem}

\begin{proof}
  First note that the operator $T$ in front of the first orthogonal space and the $T^{-1}$ attached in the first component of the inner product corresponding to $\oplus_{T^{-1}}$ cancel out and give the standard Helmholtz decomposition with $T=\mathrm{id}$.
  The proof of this standard decomposition is simply an application of the orthogonal decomposition of $\Lp{2}(\Omega)$ into $\cl{\ran L}$ and $\ker L^{\ast}$, where $L$ is the densely defined, linear operator $\cgrad$, i.e., the gradient with Dirichlet boundary condition ($\grad$ restricted to $\cH^{1}(\Omega) \subseteq \Lp{2}(\Omega)$).

  Hence, we only have to verify closedness of $\ran L$. Let $\{u_{n}\}_{n}\subseteq \ran L$ with $u_{n}=Lv_{n}$ for some $v_{n}\in \Lp{2}(\Omega)$ converge to some $u\in \Lp{2}(\Omega)$. Convergence of $(u_{n})_{n}$ implies boundedness of $(Lv_{n})_{n}$ in $\Lp{2}(\Omega)$. Since $\norm{\grad \argdot}_{\Lp{2}(\Omega)}$ and $\norm{\argdot}_{\soboH^{1}(\Omega)}$ are equivalent on $\cH^{1}(\Omega)$ by virtue of the standard Friedrichs/Poincar\'e inequality, this implies boundedness of $(v_{n})_{n}$ as a sequence in $\soboH^{1}(\Omega)$. By the Rellich--Kondrachov theorem (\cite[Thm.~12.30]{Leoni2017}), there exists a strongly in $\Lp{2}(\Omega)$ convergent subsequence $(v_{n_{k}})_{k}$ of $(v_{n})_{n}$. Appealing to the closedness of $L$, we obtain that $u=\lim_{k\to \infty}u_{n_{k}} = \lim_{k\to \infty}L v_{n_{k}} = L\lim_{k\to \infty}v_{n_{k}}$, hence $u\in \ran L$.
\end{proof}

\subsection{Boundary triples}
\label{subsec:BdyTriples}
Boundary triples were originally developed for symmetric operators, cf.~\cite{GorGor1991,BeHaSn2020}, but they can be equivalently defined in the skew-symmetric case, cf.~\cite{Wegner2017} and \cite[Ch.~2.4]{Sk-Phd2021}. This is a simple consequence of multiplying the equations by the imaginary unit $\iu$.

Boundary triples are certainly not the only tool one can apply to model (dissipative) boundary conditions. Other viable concepts include m-boundary tuples (cf.~\cite{EllerKarabash2022}), which additionally require a pivot space for the dual pair, the recent notion of {\em generalized} boundary triples (cf.~\cite{Behrndt2025}), as well as boundary quadruples (cf.~\cite{Arendt2023}) and boundary systems (cf.~\cite{Waurick2015}), which exist even if a skew-adjoint extension of the operator does not exist. However, all these notions are equivalent to the concept of boundary triples given strong enough assumptions (cf.~\cite{Waurick2018} for the equivalence of boundary systems and boundary triples). In our context, it suffices to consider boundary triples though.  We will present the definition used in \cite{Sk-Phd2021}, which allows a dual pair instead of a single (identified) Hilbert space as a boundary space.

\begin{definition}
  \label{def:BdyTriple}
  Let $A_{0}$ be a densely defined, skew-symmetric and closed operator on a Hilbert space $H$. A \emph{boundary triple} for $A_{0}^{\ast}$ is a triple $\bigl((\mathcal{B}_{+},\mathcal{B}_{-}),B_{1},B_{2}\bigr)$ consisting of a complete dual pair\footnote{\emph{Complete dual pair} simply means that that the spaces are dual to each other, but we do not identify them by means of the Riesz isomorphism in the case of a Hilbert space.} $(\mathcal{B}_{+},\mathcal{B}_{-})$ and two operators $B_{1}\colon\dom(A_{0}^{\ast})\to \mathcal{B}_{+}$ and $B_{2}\colon\dom(A_{0}^{\ast})\to \mathcal{B}_{-}$ such that
  \begin{enumerate}[leftmargin=5ex]
    \item the map
          \begin{align*}
            B=\begin{pmatrix}B_{1}\\B_{2}\end{pmatrix}\colon
            \left\{
            \begin{array}{rcl}
              \dom(A_{0}^{\ast}) & \to & \mathcal{B}_{+}\times\mathcal{B}_{-}, \\
              x & \mapsto & \begin{pmatrix}B_{1}x\\B_{2}x\end{pmatrix},
            \end{array}
            \right.
          \end{align*}
          is onto and
    \item\label{item:abstract-green} the following abstract Green identity holds for all $x,y\in\dom(A_{0}^{\ast})$:
          \begin{equation*}
            \scprod{A_{0}^{\ast}x}{y}_{X} + \scprod{x}{A_{0}^{\ast}y}_{X} = \scprod{B_{1}x}{B_{2}y}_{\mathcal{B}_{+},\mathcal{B}_{-}} + \scprod{B_{2}x}{B_{1}y}_{\mathcal{B}_{-},\mathcal{B}_{+}}.
          \end{equation*}
  \end{enumerate}
\end{definition}

The abstract Green identity~\ref{item:abstract-green} formalizes an integration by parts formula, which is the reason for the designation boundary triple.

A boundary triple enables us to parameterize all boundary conditions such that $A_{0}\adjun$ restricted to all elements of $\dom A_{0}\adjun$, that satisfy the boundary condition, is maximally dissipative. In our situation we will obtain a boundary triple that involves the trace spaces $\cH^{\frac{1}{2}}(\Gamma)$ and $\soboH^{-\frac{1}{2}}(\Gamma)$, but we would rather formulate boundary conditions in the pivot space $\Lp{2}(\Gamma)$. Verifying dissipativity is usually straightforward, but the maximality can be tricky.
In the next section, we will arrive at a situation in which $((\mathcal{B}_{-},\mathcal{B_{+}}),B_{1},B_{2})$ is a boundary triple\footnote{The order of $(\mathcal{B}_{-},\mathcal{B_{+}})$ is swapped on purpose, because later we will use $(\soboH^{-\frac{1}{2}}(\tilde{\Gamma}),\cH^{\frac{1}{2}}(\tilde{\Gamma}))$.}, where the duality of $(\mathcal{B}_{+},\mathcal{B_{-}})$ is induced by a pivot space $\mathcal{B}_{0}$. Moreover, $\mathcal{B_{+}} \subseteq \mathcal{B}_{0} \subseteq \mathcal{B}_{-}$ forms a Gelfand triple. We want to formulate Robin type boundary conditions of the form
\begin{equation*}
  B_{1}x + \Theta B_{2}x = 0,
\end{equation*}
where $\Theta \in \Lb(\mathcal{B}_{0})$. To formulate this rigorously, we have to take the embedding mappings $j_{+} \colon \mathcal{B}_{+} \to \mathcal{B}_{0}$ and $j_{-} \colon \mathcal{B}_{0} \to \mathcal{B}_{-}$ into account, i.e.,
\begin{equation*}
  j_{-}^{-1} B_{1} x + \Theta j_{+} B_{2}x = 0,
\end{equation*}
or equivalently
\begin{equation*}
  B_{1} x = - j_{-} \Theta j_{+} B_{2}x.
\end{equation*}
The corresponding operator is
\begin{align}
  \label{eq:DefATheta}
  \begin{aligned}
  A_{\Theta} &\coloneq A_{0}\adjun \big\vert_{\dom A_{\Theta}}, \\
    \dom A_{\Theta} &\coloneq \dset{x \in \dom A_{0}\adjun}{B_{1} x = - j_{-} \Theta j_{+} B_{2}x}.
  \end{aligned}
\end{align}
We denote the induced operator between $\mathcal{B}_{+}$ and $\mathcal{B}_{-}$ by $\hat{\Theta}$, i.e.,
\begin{equation*}
  \hat{\Theta} \coloneq - j_{-} \Theta j_{+}.
\end{equation*}
By the theory of boundary triples, $A_{\Theta}$ is maximally dissipative if and only if $\hat{\Theta}$ is maximally dissipative, see, e.g., \cite[Prop.~2.4.10]{Sk-Phd2021} or \cite[Cor.~2.1.4]{BeHaSn2020}.

\begin{proposition}
  \label{th:AThetaMDiss}
  Let $(\mathcal{B}_{+}, \mathcal{B}_{0}, \mathcal{B}_{-})$ be a Gelfand triple, $((\mathcal{B}_{-},\mathcal{B_{+}}),B_{1},B_{2})$ be a boundary triple for $A_{0}\adjun$ and $\Theta \in \Lb(\mathcal{B}_{0})$.
  Then $A_{\Theta}$ is maximally dissipative if $\Theta$ is positive (semi-definite), i.e., $\scprod{\Theta h}{h}_{\mathcal{B}_{0}} \geq 0$ for all $h \in \mathcal{B}_{0}$.
\end{proposition}

Note that instead of positivity we could have asked for accretivity, i.e., $\Re \scprod{\Theta h}{h}_{\mathcal{B}_{0}} \geq 0$ for all $h \in \mathcal{B}_{0}$.

\begin{proof}
  Note that $A_{\Theta}$ is maximally dissipative if $\hat{\Theta} \coloneq -j_{-} \Theta j_{+}$ is maximally dissipative.
  Hence, it is sufficient to show (the stronger assertion) that $\hat{\Theta}$ is self-adjoint and negative.

  Since $j_{+}$ and $j_{-}$ are continuous embeddings we have $\hat{\Theta} \in \Lb(\mathcal{B}_{+},\mathcal{B}_{-})$. Moreover, $j_{+}\adjun = j_{-}$ and $j_{-}\adjun = j_{+}$.\footnote{We regard the adjoint w.r.t.\ the dual pair $(\mathcal{B}_{+},\mathcal{B}_{-})$, i.e., if $A \colon \mathcal{B}_{+} \to H$, then $A\adjun\colon H \to \mathcal{B}_{-}$ and analogously if $\mathcal{B}_{+}$ is the codomain. It is the reverse way around for $\mathcal{B}_{-}$.}
  This implies
  \begin{align*}
    \Re \dualprod{\hat{\Theta} x}{x}_{\mathcal{B}_{-},\mathcal{B}_{+}}
    = -\Re \scprod{\Theta j_{+}x}{j_{+}x}_{\mathcal{B}_{0}} \leq 0.
  \end{align*}
  Furthermore we have
  \begin{equation*}
    \hat{\Theta}\adjun = - (j_{-} \Theta j_{+})\adjun
    = - j_{+}\adjun \Theta\adjun j_{-}\adjun
    = - j_{-} \Theta j_{+}
    = \hat{\Theta},
  \end{equation*}
  which implies that $\hat{\Theta}$ is self-adjoint and therefore maximally dissipative. Hence, $\hat{\Theta}^{-1}$ is maximally dissipative (in the sense of linear relations). Finally, we can apply \cite[Prop.~2.4.10]{Sk-Phd2021}, which deduces the maximal dissipativity of $A_{\Theta}$ from the maximal dissipativity of $\hat{\Theta}^{-1}$.\footnote{We point out that \cite{Sk-Phd2021} uses a different notation. The operator $A_{\Theta}$ corresponds to $A_{\hat{\Theta}^{-1}}$ in the notation of \cite{Sk-Phd2021}.}
\end{proof}


\section{Well-posedness}%
\label{sec:well-posedness}

In \cite{KurZwa2015} the existence of a boundary triple associated to the wave equation in the Dirac representation was shown.
In this section, we show that we can also associate a boundary triple to the wave equation in the Lagrange representation.
This will allow us to parameterize all maximally dissipative boundary conditions.

Recall the decomposition of $\partial\Omega$ into $\Gamma_{0}$, $\Gamma_{1}$, $\Gamma_{2}$ $\Gamma_{3}$ and $\Gamma_{4}$. Since the case where $\Gamma_{0}$ is equal to $\partial\Omega$ simply corresponds to the wave equation with Dirichlet boundary, which is well-studied, we exclude that case. In particular, we will see that the boundary conditions \eqref{eq:boundary-conditions} represent one of these maximally dissipative boundary conditions.
For the following, we only have to distinguish between $\Gamma_{0}$ and the combined other parts of the boundary.
Hence, let us abbreviate $\tilde{\Gamma}\coloneq \partial\Omega\setminus\cl{\Gamma_{0}}$, which is essentially the same as $\Gamma_{1} \cup \Gamma_{2} \cup \Gamma_{3} \cup \Gamma_{4}$.

\subsection*{Reformulation}
First we want to transfer the wave equation into the formalism for boundary triples. As state space we choose
\begin{equation*}
  X \coloneq \cH^{1}_{\Gamma_{0}}(\Omega) \times \Lp{2}(\Omega).
\end{equation*}
We use the equivalent norm for $\soboH^{1}(\Omega)$ given by the Friedrichs/Poincar\'e inequality \eqref{th:Poincare} and define
\begin{equation*}
  \scprod{x}{y}_{X} \coloneq \scprod[\big]{\tfrac{1}{\rho}x_{2}}{y_{2}}_{\Lp{2}(\Omega)} + \scprod{T \grad x_{1}}{\grad y_{1}}_{\Lp{2}(\Omega)} + \scprod{k_{1} \boundtr x_{1}}{\boundtr y_{1}}_{\Lp{2}(\tilde{\Gamma})}\text{.}
\end{equation*}
Equivalence to the canonical inner product on $X$ is assured by virtue of the Friedrichs/Poincar\'e inequality and the assumptions on $\rho$, $T$ and $k_{1}$. The inner product on $X$ is the sum of the energy inner product favored in port-Hamiltonian formulations plus the inner product on the boundary, which is necessary to prevent constant functions in the second argument to have vanishing norm. We arrive (informally) at the differential operator $A$ describing the wave equation \eqref{FullWE} by applying the classical substitutions $x_{1} = \rho \partial_{t} w$ and $x_{2} = w$ and obtain:
\begin{equation*}
  \partial_{t} \begin{pmatrix}x_{1} \\ x_{2}\end{pmatrix}
  =
  \underbrace{%
    \begin{pmatrix}
      0 & \idop  \\
      -\idop & 0
    \end{pmatrix}
    \begin{pmatrix}
      -\div T\grad & 0 \\
      0 & \frac{1}{\rho}
    \end{pmatrix}%
  }_{= \mathrlap{A}}
  \begin{pmatrix}x_{1} \\ x_{2} \end{pmatrix}
  +
  \begin{pmatrix}0 & 0 \\ -a & -b \frac{1}{\rho} \end{pmatrix}
  \begin{pmatrix}x_{1} \\ x_{2}\end{pmatrix}.
\end{equation*}
Formally we define:
\begin{equation}
  \label{eq:wave-diff-operator}
  A\colon X\supseteq \dom(A) \to  X, \quad
  \begin{pmatrix}x_{1}\\x_{2}\end{pmatrix} \mapsto \begin{pmatrix}\frac{1}{\rho}x_{2} \\ \div T\grad x_{1}\end{pmatrix}\,\text{,}
\end{equation}
where
\begin{equation*}
  \dom A \coloneq \dset*{\begin{pmatrix} x_{1} \\ x_{2} \end{pmatrix} \in X}{\tfrac{1}{\rho} x_{2} \in \cH^{1}_{\Gamma_{0}}(\Omega)\ \text{and}\ T\grad x_{1} \in \soboH(\div,\Omega)}
  \text{.}
\end{equation*}

\begin{lemma}
  $A$ is a closed operator.
\end{lemma}

\begin{proof}
  Let $(\begin{psmallmatrix} x_{n} \\ y_{n}\end{psmallmatrix})_{n\in\N}$ be a sequence in $\dom A$ that converges to $\begin{psmallmatrix}x \\ y\end{psmallmatrix} \in X$ w.r.t.\ $\norm{\argdot}_{X}$ such that also $(A \begin{psmallmatrix}x_{n} \\ y_{n}\end{psmallmatrix})_{n\in\N}$ converges to $\begin{psmallmatrix} v \\ w \end{psmallmatrix} \in X$ w.r.t.\ $\norm{\argdot}_{X}$. Note that $A \begin{psmallmatrix}x_{n} \\ y_{n}\end{psmallmatrix} = \begin{psmallmatrix}\frac{1}{\rho} y_{n} \\ \div T \grad x_{n}\end{psmallmatrix}$ implies that $(\frac{1}{\rho} y_{n})_{n}$ converges to $v$ w.r.t.\ $\norm{\argdot}_{\soboH^{1}(\Omega)}$ and therefore $v=\frac{1}{\rho} y$. Since $(x_{n})_{n}$ converges to $x$ w.r.t.\ $\norm{\argdot}_{\soboH^{1}(\Omega)}$, we conclude that $(T \grad x_{n})_{n}$ converges to $T \grad x$ w.r.t.\ $\norm{\argdot}_{\Lp{2}(\Omega)}$. Hence, the closedness of $\div$ implies $w = \div T \grad x$, which shows the closedness of $A$.
\end{proof}

$A$ is what we have called $A_{0}^{\ast}$ in \Cref{def:BdyTriple} for an (at this point still undefined) operator $A_{0}$.\footnote{In fact, we will pretermit the definition of $A_{0}$ entirely.} As outlined in \Cref{sec:preliminaries}, we want to attach boundary conditions to $A$ by identifying a suitable boundary triple and then define $A_{\Theta}$. For our boundary operators we define:
\begingroup
\renewcommand{\quad}{\mspace{10mu}}
\begin{equation}
  \label{eq:boundary-operators}
  B_{1}\colon
  \left\{
    \begin{array}{rcl}
      \dom A & \to & \soboH^{-\frac{1}{2}}(\tilde{\Gamma}), \\
      x & \mapsto & k_{1} \boundtr x_{1} + \normaltr T \grad x_{1},
    \end{array}
  \right.
  \mspace{-10mu}\quad\text{and}\quad
  B_{2}\colon
  \left\{
    \begin{array}{rcl}
      \dom A & \to & \cH^{\frac{1}{2}}(\tilde{\Gamma}), \\
      x & \mapsto & \boundtr \tfrac{1}{\rho} x_{2}.
    \end{array}
  \right.
\end{equation}
\endgroup
The ultimate goal of this section is to show that $\bigl(\bigl(\soboH^{-\frac{1}{2}}(\tilde{\Gamma}), \cH^{\frac{1}{2}}(\tilde{\Gamma})\bigr), B_{1}, B_{2}\bigr)$ is a boundary triple for $A$. This will allow us to attach the boundary conditions \eqref{eq:boundary-conditions} to $A$ via a relation $\Theta$ essentially given by $k_{2}$ as described in \Cref{subsec:BdyTriples}. From there, well-posedness of problem \eqref{FullWE} will easily follow. According to \Cref{def:BdyTriple}, for a boundary triple we have to verify:
\begin{enumerate}[leftmargin=5ex]
  \item $B \coloneq \begin{psmallmatrix} B_{1} \\ B_{2}\end{psmallmatrix}$ is onto.
  \item $A$ and $B$ satisfy an abstract Green identity.
\end{enumerate}
To verify surjectivity of $B$ it suffices to show that the individual $B_{1}$ and $B_{2}$ are surjective, since they act on different components of $x = \begin{psmallmatrix} x_{1} \\ x_{2}\end{psmallmatrix}$. We first establish a useful result about solutions of the (generalized) Dirichlet problem with boundary condition provided by $B_{1}$. The following proposition is a slight modification of \cite[Thm.~5.5]{ToChTo2019}:

\begin{proposition}
  \label{th:B1-surjective-on-ker-divTgrad}
  For any $g \in \soboH^{-\frac{1}{2}}(\tilde{\Gamma})$ there exists $w \in \dset{f \in \cH^{1}_{\Gamma_{0}}(\Omega)}{T \grad f \in \soboH(\div,\Omega)}$ such that
  \begin{alignat*}{3}
    \div T \grad w &= 0 &\quad \text{in}\ \Omega, \\
    \normaltr T \grad w + k_{1} \boundtr w &= g &\quad \text{on}\ \tilde{\Gamma}.
  \end{alignat*}
  In particular, $w \coloneq \boundtr\adjun g$ given by the adjoint
  \begin{equation*}
    \boundtr\adjun \colon \soboH^{-\frac{1}{2}}(\tilde{\Gamma}) \to \cH^{1}_{\Gamma_{0}}(\Omega)
    \quad\text{of}\quad
    \boundtr \colon \cH^{1}_{\Gamma_{0}}(\Omega) \to \cH^{\frac{1}{2}}(\tilde{\Gamma}),
  \end{equation*}
  where $\cH^{1}_{\Gamma_{0}}(\Omega)$ is equipped with
  \begin{equation*}
    \scprod{f}{g}_{\grad,\boundtr} \coloneq \scprod{T \grad f}{\grad g}_{\Lp{2}(\Omega)} + \scprod{k_{1} \boundtr f}{\boundtr g}_{\Lp{2}(\tilde{\Gamma})},
  \end{equation*}
  is a solution.
\end{proposition}

\begin{proof}
  Note that $\boundtr\colon \cH^{1}_{\Gamma_{0}}(\Omega) \to \cH^{\frac{1}{2}}(\tilde{\Gamma})$ is bounded. Thus, its adjoint is defined on all of $\soboH^{-\frac{1}{2}}(\tilde{\Gamma})$ and is bounded. Moreover, $\boundtr\adjun g \in \cH^{1}_{\Gamma_{0}}(\Omega)$ as this is the range of $\boundtr\adjun$. By definition of the adjoint, for $v \in \cH^{1}_{\Gamma_{0}}(\Omega)$  we have
  \begin{align*}
    \dualprod{g}{\boundtr v}_{\mp \frac{1}{2}}
    = \scprod{\boundtr\adjun g}{v}_{\grad,\boundtr}
    = \scprod{T \grad \boundtr\adjun g}{\grad v}_{\Lp{2}(\Omega)} + \scprod{k_{1} \boundtr \boundtr\adjun g}{\boundtr v}_{\Lp{2}(\tilde{\Gamma})}.
  \end{align*}
  If we choose $v \in \Cc(\Omega)$, we obtain
  \begin{equation*}
    0 = \scprod{T\grad \boundtr\adjun g}{\grad v}_{\Lp{2}(\Omega)},
  \end{equation*}
  which implies $T \grad \boundtr\adjun g \in \soboH(\div,\Omega)$ and $\div T \grad \boundtr\adjun g = 0$. Hence, for arbitrary $v \in \cH^{1}_{\Gamma_{0}}(\Omega)$ integration by parts gives
  \begin{align*}
    \dualprod{g}{\boundtr v}_{\mp \frac{1}{2}}
    &= \dualprod{\normaltr T \grad \boundtr\adjun g}{\boundtr v}_{\mp \frac{1}{2}}
      + \dualprod{k_{1} \boundtr \boundtr\adjun g}{\boundtr v}_{\mp \frac{1}{2}}
    \\
    &= \dualprod{\normaltr T \grad \boundtr\adjun g + k_{1} \boundtr \boundtr\adjun g}{\boundtr v}_{\mp \frac{1}{2}}.
  \end{align*}
  Hence $w \coloneq \boundtr\adjun g$ is a solution.
\end{proof}

\begin{corollary}
  \label{th:B-surjective}
  $B$ is surjective.
\end{corollary}
\begin{proof}
  \Cref{th:B1-surjective-on-ker-divTgrad} shows that $B_{1}$ is surjective and $B_{2}$ is surjective as a consequence of the surjectivity of the Dirichlet trace operator. Since $B_{1}$ and $B_{2}$ act on different components, we obtain the surjectivity of $B$.
\end{proof}

For the second step we observe that integration by parts for the Dirichlet and normal trace operators (\Cref{th:PI}) gives rise to an abstract Green identity:

\begin{proposition}[Green identity]
  \label{th:green-identity}
  For all $x,y \in \dom A$ the following identity holds:
  \begin{equation*}
    \scprod{A x}{y} + \scprod{x}{Ay} = \scprod{B_{1}x}{B_{2} y} + \scprod{B_{2}x}{B_{1}y}.
  \end{equation*}
\end{proposition}

\begin{proof}
  Let $x,y\in \dom(A)$. We can calculate:
    \begin{align*}
    \MoveEqLeft
    \scprod{A x}{y}_{X}+\scprod{x}{A y}_{X} \\
    &=
      \scprod[\bigg]{\begin{pmatrix}\tfrac{1}{\rho}x_{2} \\ \div T\grad x_{1}\end{pmatrix}} {\begin{pmatrix}y_{1}\\y_{2}\end{pmatrix}}_{X}
      + \scprod[\bigg]{\begin{pmatrix}x_{1}\\x_{2}\end{pmatrix}}
      {\begin{pmatrix}\tfrac{1}{\rho}y_{2} \\ \div T\grad y_{1} \end{pmatrix}}_{X}\\
    \intertext{By the definition of the inner product in $X$ we have:}
    &=
      \scprod{\tfrac{1}{\rho}\div T\grad x_{1}}{y_{2}}_{\Lp{2}(\Omega)}
      + \scprod{T\grad \tfrac{1}{\rho}x_{2}}{\grad y_{1}}_{\Lp{2}(\Omega)}
      + \scprod{k_{1}\boundtr \tfrac{1}{\rho}x_{2}}{\boundtr y_{1}}_{\Lp{2}(\tilde{\Gamma})}
      \\
    &\quad
      + \scprod{x_{2}}{\tfrac{1}{\rho}\div T\grad y_{1}}_{\Lp{2}(\Omega)}
      + \scprod{\grad x_{1}}{T\grad\tfrac{1}{\rho}y_{2}}_{\Lp{2}(\Omega)}
      + \scprod{\boundtr x_{1}}{k_{1}\boundtr \tfrac{1}{\rho}y_{2}}_{\Lp{2}(\tilde{\Gamma})}
      \\[0.5ex]
    \intertext{Applying integration by parts on the $\div T \grad$ operators gives {(we suppress the index for the inner product in the following and just distinguish between inner products and dual pairings)}:}
    &=
      - \scprod{T\grad x_{1}}{\grad\tfrac{1}{\rho}y_{2}}
      + \dualprod{\normaltr T\grad x_{1}}{\boundtr \tfrac{1}{\rho} y_{2}}_{\mp \frac{1}{2}}
      + \scprod{T\grad \tfrac{1}{\rho}x_{2}}{\grad y_{1}}
      + \scprod{k_{1}\boundtr \tfrac{1}{\rho}x_{2}}{\boundtr y_{1}}
    \\
    &\quad
      - \scprod{\grad\tfrac{1}{\rho}x_{2}}{T\grad y_{1}}
      + \dualprod{\boundtr \tfrac{1}{\rho} x_{2}}{\normaltr T\grad y_{1}}_{\pm \frac{1}{2}}
      + \scprod{\grad x_{1}}{T\grad\tfrac{1}{\rho}y_{2}}
      + \scprod{\boundtr x_{1}}{k_{1} \boundtr \tfrac{1}{\rho}y_{2}}
      \\[0.5ex]
    \intertext{This simplifies to}
    &= \dualprod{\normaltr T\grad x_{1}}{\boundtr \tfrac{1}{\rho} y_{2}}_{\mp \frac{1}{2}}
      + \scprod{k_{1} \boundtr \tfrac{1}{\rho}x_{2}}{\boundtr y_{1}}
    \\
    &\quad
      + \dualprod{\boundtr \tfrac{1}{\rho} x_{2}}{\normaltr T\grad y_{1}}_{\pm \frac{1}{2}}
      + \scprod{\boundtr x_{1}}{k_{1}\boundtr \tfrac{1}{\rho}y_{2}}
    \\[0.5ex]
    &= \dualprod{k_{1}\boundtr x_{1} + \normaltr T\grad x_{1}}{\boundtr \tfrac{1}{\rho} y_{2}}_{\mp \frac{1}{2}}
      + \dualprod{\boundtr \tfrac{1}{\rho} x_{2}}{k_{1} \boundtr y_{1} + \normaltr T\grad y_{1}}_{\pm \frac{1}{2}}
    \\[0.5ex]
    &= \dualprod{B_{1}x}{B_{2} y}_{\mp \frac{1}{2}} + \dualprod{B_{2}x}{B_{1}y}_{\pm \frac{1}{2}}.
      \tag*{\qedhere}
  \end{align*}
\end{proof}

To avoid having to define $A_{0}$ (from \Cref{def:BdyTriple}) itself and having to verify that $A_{0}^{\ast}=A$, we will simply prove $A^{\ast} \subseteq - A$. This way, $A_{0} \coloneq - A^{\ast}$ has all necessary properties of $A_{0}$, i.e., skew-symmetry, dense domain and closedness.\footnote{Alternatively, one can set $A_{0}\coloneq -A$ with $\dom A_{0}\coloneq \ker B_{1} \cap \ker B_{2}$, cf.~\cite[Lem.~2.4.5]{Sk-Phd2021}. However, this requires careful handling to avoid circular arguments.}

\begin{proposition}
  $A^{\ast} \subseteq - A$ and $A$ is densely defined.
\end{proposition}

Note that the following proof must be formulated in the language of linear relations, since we do not know a priori whether the adjoint of $A$ is single-valued---which is equivalent to $A$ being densely defined.
This has the useful side effect that the density of $A$'s domain follows automatically.

\begin{proof}
  Let $(y,z) \in A\adjun$ (where $y = \begin{psmallmatrix} y_{1} \\ y_{2}\end{psmallmatrix} \in X$ and $z = \begin{psmallmatrix} z_{1} \\ z_{2} \end{psmallmatrix} \in X$). Then we have for $x=\begin{psmallmatrix}x_{1}\\x_{2}\end{psmallmatrix} \in \dom A$
  \begin{equation*}
    \scprod{Ax}{y}_{X}
    = \scprod{x}{z}_{X},
  \end{equation*}
  or equivalently
  \begin{multline}
    \label{eq:adjoint-of-A-equation}
    \scprod{\tfrac{1}{\rho}\div T \grad x_{1}}{y_{2}}_{\Lp{2}(\Omega)}
    + \scprod{T\grad \tfrac{1}{\rho} x_{2}}{\grad y_{1}}_{\Lp{2}(\Omega)} + \scprod{k_{1} \boundtr \tfrac{1}{\rho}x_{2}}{\boundtr y_{1}}_{\Lp{2}(\tilde{\Gamma})}
    \\
    = \scprod{\tfrac{1}{\rho}x_{2}}{z_{2}}_{\Lp{2}(\Omega)} + \scprod{T \grad x_{1}}{\grad z_{1}}_{\Lp{2}(\Omega)}
    + \scprod{k_{1} \boundtr x_{1}}{\boundtr z_{1}}_{\Lp{2}(\tilde{\Gamma})}.
  \end{multline}
  Let $x_{2} \in \Cc(\Omega)$. Then we choose $x = \begin{psmallmatrix} 0 \\ \rho x_{2}\end{psmallmatrix}$, which is in $\dom A$, and obtain
  \begin{equation*}
    \scprod{\grad x_{2}}{T\grad y_{1}}_{\Lp{2}(\Omega)}
    = \scprod{x_{2}}{z_{2}}_{\Lp{2}(\Omega)}.
  \end{equation*}
  This holds true for all $x_{2}\in \Cc(\Omega)$; hence $T\grad y_{1} \in \soboH(\div,\Omega)$ and $z_{2} = -\div T \grad y_{1}$.

  For arbitrary $g \in \soboH^{-\frac{1}{2}}(\tilde{\Gamma})$ we let $x_{1} = \boundtr\adjun g$.  Appealing to \Cref{th:B1-surjective-on-ker-divTgrad} we conclude $B_{1} \begin{psmallmatrix} x_{1} \\ 0\end{psmallmatrix} = g$ and $\div T \grad x_{1} = 0$. This assures $x = \begin{psmallmatrix}x_{1} \\ 0\end{psmallmatrix} \in \dom A$, which allows us to calculate
  \begin{align*}
    0 = \scprod{\tfrac{1}{\rho}\,\smash[b]{\underbrace{\div T \grad x_{1}}_{=\mathrlap{0}}}}{y_{2}}_{\Lp{2}(\Omega)}
    &\stackrel{\mathclap{\eqref{eq:adjoint-of-A-equation}}}{=}
      \scprod{T \grad x_{1}}{\grad z_{1}}_{\Lp{2}(\Omega)} + \scprod{k_{1} \boundtr x_{1}}{\boundtr z_{1}}_{\Lp{2}(\tilde{\Gamma})} \\
    &= 0 + \dualprod{\normaltr T \grad x_{1}}{\boundtr z_{1}}_{\mp \frac{1}{2}}
      + \scprod{k_{1} \boundtr x_{1}}{\boundtr z_{1}}_{\Lp{2}(\tilde{\Gamma})} \\
    &= \dualprod{g}{\boundtr z_{1}}_{\mp \frac{1}{2}}\text{.}
  \end{align*}
  Since $g \in \soboH^{-\frac{1}{2}}(\tilde{\Gamma})$ was arbitrary, we conclude $\boundtr z_{1} = 0$.
  Appealing to the Helmholtz decomposition $\Lp{2}(\Omega) = T \grad \cH^{1}(\Omega) \oplus_{T^{-1}} \ker \div$ from \Cref{th:Helmholtz}, we know that $\ran(\div T \cgrad) = \ran(\div) = \Lp{2}(\Omega)$. For $x = \begin{psmallmatrix}x_{1} \\ 0\end{psmallmatrix}$ for arbitrary $x_{1} \in \dset{f \in \cH^{1}(\Omega)}{T \grad x_{1} \in \soboH(\div,\Omega)}$ we have $x \in \dom A$ and by \eqref{eq:adjoint-of-A-equation}
  \begin{equation*}
    \scprod{\tfrac{1}{\rho} \div T \grad x_{1}}{y_{2}}_{\Lp{2}(\Omega)} = \scprod{T \grad x_{1}}{\grad z_{1}}_{\Lp{2}(\Omega)} + 0 = -\scprod{\div T \grad x_{1}}{z_{1}}_{\Lp{2}(\Omega)}.
  \end{equation*}
  Thus, surjectivity of $\div T \cgrad$ implies $z_{1} = - \tfrac{1}{\rho} y_{2}$, in particular $\tfrac{1}{\rho} y_{2} \in \cH^{1}_{\Gamma_{0}}(\Omega)$ (because $z_{1} \in \cH^{1}_{\Gamma_{0}}(\Omega)$ by assumption).

  Altogether we have shown $y \in \dom A$ and $z = -A y$, i.e., $A\adjun \subseteq - A$.

  Note that $A\adjun$ is not multi-valued, since $A\adjun \subseteq - A$. This implies $(\dom A)^{\perp} = \set{0}$, cf. \cite[Lem.\ 2.2.8]{Sk-Phd2021}, or equivalently $\cl{\dom A} = X$.
\end{proof}

With this, all initially outlined steps for the verification of a boundary triple are in place.

\begin{theorem}[Boundary triple]
  \label{th:BdyTriple}
  $\bigl(\soboH^{-\frac{1}{2}}(\tilde{\Gamma}), \cH^{\frac{1}{2}}(\tilde{\Gamma}), B_{1}, B_{2}\bigr)$ is a boundary triple for $A$.
\end{theorem}

\begin{proof}
  By definition of a boundary triple, we have to check that an abstract Green identity holds (verified in \Cref{th:green-identity}) and that $B$ is onto (verified in \Cref{th:B-surjective}).
\end{proof}

The remainder of this article heavily relies on exploiting \Cref{th:BdyTriple}. The first step is to establish well-posedness of the initial problem \eqref{FullWE}.


\subsection*{Well-posedness of \Cref{FullWE}}
For this we simply have to specify the boundary relation $\Theta$ from \eqref{eq:DefATheta}. We set
\begin{equation*}
  \Theta \coloneq k_{2} \in \Lb(\Lp{2}(\tilde{\Gamma}))\ \text{(as an operator)}.
\end{equation*}
Then the collection of boundary conditions from (\ref{eq:boundary-conditions}) can be written as $B_{1}x + k_{2}B_{2}x = 0$ on $\tilde{\Gamma}$.
In particular,
\begin{equation}\label{eq:def-A-Theta}
  B_{1}x + k_{2} B_{2} x = 0
  \quad\iff\quad
  k_{1} w + \nu \cdot T \grad w + k_{2} \partial_{t} w = 0.
\end{equation}
If one replaces $k_{1},k_{2}\in \mathcal{L}(\Lp{2}(\tilde{\Gamma}))$ with suitable multiplication operators arising from functions by the same name, the equation to the right contains the second to the fifth condition of \eqref{eq:boundary-conditions} in one line by taking into account that $k_{1}$ and $k_{2}$ may vanish on certain parts of $\tilde{\Gamma}$. The condition $w = 0$ on $\Gamma_{0}$ is already encoded in the domain of the operator $A$ and completes the set of equations from (\ref{eq:boundary-conditions}). Hence, the operator that encodes all boundary conditions is $A_{\Theta}$ defined by \eqref{eq:DefATheta} (with $A_{0}\adjun = A$ from \eqref{eq:wave-diff-operator} and $\Theta = k_{2}$), i.e.,
\begin{align*}
  A_{\Theta} = \begin{pmatrix} 0 & \frac{1}{\rho} \\ \div T \grad & 0 \end{pmatrix}, \quad
  \dom A_{\Theta} = \dset*{\begin{pmatrix}x_{1} \\ x_{2}\end{pmatrix} \in \dom A}{B_{1} x + k_{2} B_{2}x = 0}.
\end{align*}

\begin{theorem}
  \label{th:AThetaIsGen}
  $A_{\Theta}$ is generator of a strongly continuous semigroup of contractions.
\end{theorem}

\begin{proof}
  By \Cref{th:BdyTriple} $\bigl((\soboH^{-\frac{1}{2}}(\tilde{\Gamma}),\cH^{\frac{1}{2}}(\tilde{\Gamma})), B_{1},B_{2}\bigr)$ is a boundary triple for $A$. Hence, \Cref{th:AThetaMDiss} and the Lumer--Philips \Cref{th:Lumer-Philips} imply the claim.
\end{proof}

\Cref{th:AThetaIsGen} assures well-posedness of the wave equation without the perturbation terms $a$ and $b$. Note that we assumed the initial conditions be in the state space $X$ already in \Cref{sec:introduction}. For the full problem \eqref{FullWE} we define the perturbation
\begin{equation*}
  S\colon X\to X,\quad \begin{pmatrix}x_{1}\\x_{2}\end{pmatrix}\mapsto\begin{pmatrix} 0 \\ - a x_{1} - \frac{1}{\rho}b x_{2} \end{pmatrix}.
\end{equation*}
For $a,b \in \Lp{\infty}(\Omega)$ we obtain a bounded linear operator. Then the differential equation from \eqref{FullWE} becomes
\begin{equation*}
  \dot{x} = (A_{\Theta} + S) x
\end{equation*}
and we have unique solvability appealing to standard perturbation theory (e.g., \cite[Thm.~1.3]{Engel2006}). One can also allow more general $a,b$ as long as $S$ remains suitably relatively $A_{\Theta}$ bounded, cf.\ \cite[Thm.~2.7]{Engel2006} for a suitable perturbation result.

\section{Stability}
\label{sec:stability}

Under mild additional assumptions we are able to show stability of solutions. For the purpose of this section, we will additionally assume that:
\begin{itemize}
  \item There exists $\Gamma \subseteq \partial\Omega \setminus \cl{\Gamma_{0}}$ open and non-empty such that $k_{2} > 0$ on $\Gamma$ (i.e., there is an open set where the wave equation is damped).

  \item $T$ and $\rho$ are Lipschitz continuous multiplication operators.
\end{itemize}

Making use of the results of the previous section, we have to show stability properties of the semigroup generated by $A_{\Theta}$, where $A_{\Theta}$ is the operator from the previous section, see \eqref{eq:def-A-Theta}, i.e., $\Theta = k_{2}$.  The notion of stabiliy we have in mind is the following one:
\begin{definition}
  \label{def:SUS}
  Let $H$ be a Hilbert space and let $\{T(t)\}_{t\geq 0} \subseteq \mathcal{L}(H)$ be a strongly continuous semigroup with generator $G$. The semigroup is called \emph{semi-uniformly stable} if there exists a continuous non-increasing function $f\colon \lbrack 0,\infty \rparen \to \lbrack 0,\infty \rparen$ satisfying $\lim_{t\to\infty}f(t)=0$, such that for every $x\in \dom G$:
  \begin{equation*}
    \lim_{t\to \infty}\norm{T(t) x}_{H}\leq f(t)\norm{x}_{\dom G}.
  \end{equation*}
\end{definition}

\begin{remark}
  The notion of semi-uniform stability is nested in between strong stability, where
  \begin{equation*}
    \lim_{t\to\infty}\norm{T(t)x}_{H}=0
  \end{equation*}
  for all $x\in H$ is demanded, and uniform stability, where
  \begin{equation*}
    \lim_{t\to\infty}\norm{T(t)}=0
  \end{equation*}
  is demanded, which for strongly continuous semigroups is equivalent to uniform exponential stability, i.e., there exists $\epsilon>0$ such that:
  \begin{equation*}
    \lim_{t\to\infty}\e^{\epsilon t}\norm{T(t)} = 0.
  \end{equation*}
\end{remark}

We rely on the following criterium from \cite[Thm.~1.1]{BaDu08}\footnote{The decay rate can be made explicit, cf.~\cite[Thm.~3.4]{ChSeTo2020} and its proof for details.}.


\begin{proposition}
  \label{CritForSUS}
  Let $(T(t))_{t\geq 0}$ be a bounded, strongly continuous semigroup with generator $G$ satisfying $\upsigma (G)\cap \iu \R=\emptyset$. Then $(T(t))_{t\geq 0}$ is semi-uniformly stable.
\end{proposition}

\Cref{CritForSUS} outlines a roadmap we can follow with the goal of showing that $A_{\Theta}$ generates a semi-uniformly stable semigroup. We will verify:
\begin{enumerate}[leftmargin=5ex]
  \item\label{item:semigroup-bounded} $(T(t))_{t\geq 0}$ is bounded.
  \item\label{item:iR-without-zero-in-resolvent-set} $\iu \lambda -A_{\Theta}$ is boundedly invertible for $\lambda \in \R\setminus \{0\}$.
  \item\label{item:zero-in-resolvent-set} $A_{\Theta}$ is boundedly invertible.
\end{enumerate}

The point \ref{item:semigroup-bounded} is clear, as \Cref{th:AThetaMDiss} states that $A_{\Theta}$ generates a contraction semigroup, in particular the semigroup is bounded. For \ref{item:iR-without-zero-in-resolvent-set} and \ref{item:zero-in-resolvent-set} we follow the approach of \cite{Jacob2021}.
We start by making the following observation:
\begin{proposition}
  \label{th:compact-embedding}
  The embedding of $\dom A_{\Theta}$ into $X$ is compact, i.e., $\dom A_{\Theta} \cpt X$.
\end{proposition}

\begin{proof}
  Let $\left(\begin{psmallmatrix} x_{n} \\ y_{n}\end{psmallmatrix}\right)_{n\in\N}$ be a bounded sequence in $\dom A$ w.r.t.\ the graph norm of $A$, i.e., there exists a $C > 0$ independent of $n\in\N$ such that
  \begin{align*}
    \norm*{\begin{psmallmatrix} x_{n} \\ y_{n} \end{psmallmatrix}}_{\dom A}^{2}
    \coloneq \norm*{\begin{psmallmatrix} x_{n} \\ y_{n} \end{psmallmatrix}}_{X}^{2} + \norm*{A \begin{psmallmatrix} x_{n} \\ y_{n} \end{psmallmatrix}}_{X}^{2}
    \leq C.
  \end{align*}
  By the definition of the norm in $X$ we obtain
  \begin{multline*}
    \norm[\big]{\rho^{-\frac{1}{2}} y_{n}}_{\Lp{2}(\Omega)}^{2} + \norm[\big]{T^{\frac{1}{2}} \grad x_{n}}_{\Lp{2}(\Omega)}^{2} + \norm[\big]{k_{1}^{\frac{1}{2}}\boundtr x_{n}}_{\Lp{2}(\tilde{\Gamma})}^{2} \\
    + \norm[\big]{\rho^{-\frac{1}{2}} \div T \grad x_{n}}^{2}_{\Lp{2}(\Omega)} + \norm[\big]{T^{\frac{1}{2}}\grad \tfrac{1}{\rho} y_{n}}^{2}_{\Lp{2}(\Omega)} + \norm[\big]{k_{1}^{\frac{1}{2}} \boundtr \tfrac{1}{\rho} y_{n}}^{2}_{\Lp{2}(\tilde{\Gamma})} \leq C\text{.}
  \end{multline*}
  This immediately implies that $(x_{n})_{n\in\N}$ and $(\tfrac{1}{\rho} y_{n})_{n\in\N}$ are bounded in $\soboH^{1}(\Omega)$.

  We have to show that $(x_{n})_{n\in\N}$ has a convergent subsequence in $\cH_{\Gamma_{0}}^{1}(\Omega)$ and $(y_{n})_{n\in\N}$ has a convergent subsequence in $\Lp{2}(\Omega)$.

  \begin{itemize}
    \item Since $(\tfrac{1}{\rho}y_{n})_{n\in\N}$ is bounded in $\soboH^{1}(\Omega)$, there exists a subsequence $(\tfrac{1}{\rho}y_{n(k)})_{k\in\N}$ that converges in $\Lp{2}(\Omega)$ appealing to the Rellich--Kondrachov theorem (\cite[Thm.~12.30]{Leoni2017}). Hence, $(\rho \tfrac{1}{\rho} y_{n(k)})_{k\in\N}$ converges in $\Lp{2}(\Omega)$ as well.

    \item Appealing to Rellich--Kondrachov again, boundedness of $(x_{n})_{n\in\N}$ in $\soboH^{1}(\Omega)$ implies existence of a subsequence that converges to $x$ in $\Lp{2}(\Omega)$. W.l.o.g.\ we pass to that subsequence.
          To avoid the introduction of several (irrelevant) multiplicative constants, we use the symbol $\lesssim$, which stands for inequality up to a multiplicative constant independent of $n$.
          Note that appealing to the boundary condition prescribed by $\Theta$, we have
          \begin{equation*}
            \norm{\normaltr T\nabla x_{n}}_{\Lp{2}(\tilde{\Gamma})} = \norm{k_{1}\boundtr x_{n} + k_{2} \boundtr \tfrac{1}{\rho}y_{n}}_{\Lp{2}(\tilde{\Gamma})} \lesssim 1\text{.}
          \end{equation*}
          Integration by parts and the Cauchy--Schwarz inequality allow us to estimate
          \begin{align*}
            \MoveEqLeft
            \norm{\grad (x_{n}- x_{m})}^{2}_{\Lp{2}(\Omega)} \\
            &\lesssim \norm[\big]{T^{\frac{1}{2}} \grad (x_{n}- x_{m})}^{2}
              = \scprod[\big]{T \grad(x_{n} - x_{m})}{\grad (x_{n} - x_{m})} \\
            &\stackrel{\mathllap{\text{Int.\ by parts}}}{=} \scprod[\big]{-\div T \grad(x_{n} - x_{m})}{x_{n} - x_{m}} + \scprod{\normaltr T\grad(x_{n}-x_{m})}{\boundtr (x_{n}-x_{m})} \\
            &\stackrel{\mathllap{\text{C.--S.\ ineq.}}}{\lesssim} \norm{x_{n} - x_{m}}_{\Lp{2}(\Omega)} + \norm{\boundtr (x_{n} - x_{m})}_{\Lp{2}(\tilde{\Gamma})}.
          \end{align*}
          Since $(x_{n})_{n\in\N}$ is bounded in $\soboH^{1}(\Omega)$, the sequence of Dirichlet traces $(\gamma_{0}x_{n})_{n\in\N}$ is bounded in $\soboH^{\frac{1}{2}}(\partial\Omega)$. Since $\soboH^{\frac{1}{2}}(\partial\Omega)$ is compactly embedded in $\Lp{2}(\partial\Omega)$ by \cite[Thm.~7.1]{DiNezza2012}, after passing to a subsequence we obtain convergence of $(\gamma_{0}x_{n})_{n\in\N}$ in $\Lp{2}(\partial\Omega)$ (and therefore also in $\Lp{2}(\tilde{\Gamma})$). Hence, the estimate shows that $(\grad x_{n})_{n\in\N}$ is a Cauchy sequence in $\Lp{2}(\Omega)$ and therefore convergent. By the closedness of the operator $\grad$, we conclude that $x$ is also the limit of $(x_{n})_{n\in\N}$ in $\soboH^{1}(\Omega)$.
          Finally, closedness of $\cH_{\Gamma_{0}}^{1}(\Omega)$ in $\soboH^{1}(\Omega)$ implies $x\in\cH_{\Gamma_{0}}^{1}(\Omega)$.
          \qedhere
  \end{itemize}
\end{proof}

This result has the consequence that we only have to investigate eigenvalues:
\begin{theorem}
  \label{onlyEV}
  $\upsigma(A_{\Theta})$ consists purely of eigenvalues.
\end{theorem}

\begin{proof}
  Let $\lambda -A_{\Theta}$ be injective. By \Cref{th:compact-embedding} all resolvent operators of $A_{\Theta}$ are compact. Hence, let $\mu$ be such that $(\mu - A_{\Theta})^{-1}$ is bijective and bounded (and compact). Then:
  \begin{align*}
    \lambda - A_{\Theta} &= (\lambda - \mu) + (\mu - A_{\Theta})\\
    (\lambda - A_{\Theta})(\mu - A_{\Theta})^{-1} &= (\lambda - \mu)(\mu - A_{\Theta})^{-1} + 1
  \end{align*}
  The right-hand side is of the form $1 + K$, where $K$ is a compact operator. It is a consequence of the theorem of Riesz--Schauder, cf.\ \cite[Thm.~6.2.1]{Werner2000}, that an operator $1+K$ is injective if and only if it is surjective. Now note that the left-hand side is injective, since $\lambda - A_{\Theta}$ is injective and the resolvent is bijective. Thus $\lambda - A_{\Theta}$ is surjective. The open mapping theorem assures that $\lambda - A_{\Theta}$ is boundedly invertible.
\end{proof}




\Cref{CritForSUS} requires us to study the spectrum of $A_{\Theta}$ on the imaginary axis. \Cref{onlyEV} makes this considerably easier, as we only have to check for eigenvalues. The following lemma gives us important information on the eigenvalues:
\begin{lemma}
  \label{th:BdyOfEV}
  If $\lambda \in \iu \R$ is an eigenvalue of $A_{\Theta}$, the corresponding eigenvectors $x = \begin{psmallmatrix} x_{1}\\x_{2}\end{psmallmatrix}$ satisfy
  \begin{equation*}
    \normaltr T\grad x_{1} + k_{1}\boundtr x_{1} = 0
    \quad\text{and}\quad
    k_{2} \boundtr \tfrac{1}{\rho} x_{2} = 0
    \quad(\text{on}\ \tilde{\Gamma}).
  \end{equation*}
\end{lemma}

\begin{proof}
  We can calculate for $\lambda \in \upsigma(A_{\Theta})$ and a corresponding eigenvector $x \in \dom A_{\Theta}$:
  \begin{align*}
    0 = \Re \scprod[\big]{\smash[b]{\underbrace{(A - \lambda)x}_{=\mathrlap{0}}}}{x}
    &= \Re \scprod{Ax}{x} - \Re \lambda \scprod{x}{x}\\
    &= \tfrac{1}{2}\bigl(\scprod{Ax}{x} + \scprod{x}{Ax}\bigr) - \Re \lambda \norm{x}^{2}\\
    &= \tfrac{1}{2}\bigl(\dualprod{B_{1}x}{B_{2}x}_{\mp \frac{1}{2}} + \dualprod{B_{2}x}{B_{1}x}_{\pm \frac{1}{2}}\bigr) - \Re \lambda \norm{x}^{2}\\
    &= \Re \dualprod{B_{1}x}{B_{2}x}_{\mp \frac{1}{2}} - \Re \lambda \norm{x}^{2}\\
    &= \Re \dualprod{\normaltr T\grad x_{1} + k_{1} \boundtr x_{1}}{\boundtr \tfrac{1}{\rho}x_{2}}_{\mp\frac{1}{2}} - \Re \lambda \norm{x}^{2}.
  \end{align*}
  If $\lambda\in \iu \R$, we obtain
  \begin{equation*}
    0 = \Re \dualprod{\normaltr T\grad x_{1} + k_{1} \boundtr x_{1}}{\boundtr \tfrac{1}{\rho}x_{2}}_{\mp\frac{1}{2}}.
  \end{equation*}
  Appealing to the boundary condition
  \begin{equation*}
    B_{1}x + \Theta B_{2}x = 0
    \quad \Longleftrightarrow \quad
    k_{1} \boundtr x_{1} + \normaltr T \grad x_{1} + k_{2} \boundtr \tfrac{1}{\rho} x_{2} = 0
  \end{equation*}
  and the fact that the duality can be written as an inner product if the arguments are in $\Lp{2}(\partial\Omega)$, we infer
  \begin{equation*}
    0 = \Re \dualprod{-k_{2}\boundtr \tfrac{1}{\rho}x_{2}}{\boundtr \tfrac{1}{\rho}x_{2}}_{\mp\frac{1}{2}}
    = -\Re \scprod{k_{2}\boundtr \tfrac{1}{\rho}x_{2}}{\boundtr \tfrac{1}{\rho}x_{2}}
    = -\norm{k_{2}^{\frac{1}{2}} \boundtr \tfrac{1}{\rho} x_{2}}^{2}\text{,}
  \end{equation*}
  which implies $k_{2}\boundtr \tfrac{1}{\rho} x_{2} = 0$. Employing the boundary condition provided by $\Theta$
  again, we deduce $\normaltr T\grad x_{1} + k_{1} \boundtr x_{1} = 0$.
\end{proof}

Note that the previous result says that any eigenvector $x = \begin{psmallmatrix} x_{1}\\x_{2}\end{psmallmatrix}$ to an eigenvalue $\iu \lambda$ satisfies $\boundtr \tfrac{1}{\rho} x_{2} = 0$ on $\supp k_{2}$, which by assumption \ref{item:assumption-boundary-operators} implies that there is a non-empty open set $\Gamma$ contained in $\supp k_{2}$ such that $\boundtr \tfrac{1}{\rho} x_{2} = 0$ on $\Gamma$.

To study purely imaginary eigenvalues, we need to investigate the equation $(\iu \lambda - A_{\Theta})u = 0$, which can be written as the system:
\begin{align}
  \begin{aligned}
    \label{eq:invSys}
    \iu \lambda u_{1} - \tfrac{1}{\rho}u_{2} &= 0 \quad\text{in}\; \soboH^{1}(\Omega), \\
    \iu \lambda u_{2} - \div T\grad u_{1} &= 0 \quad\text{in}\; \Lp{2}(\Omega).
  \end{aligned}
\end{align}
We consider the cases $\lambda \in \R \setminus\set{0}$ and $\lambda = 0$.

\begin{lemma}
  \label{th:imaginary-axis-in-resolvent}
  $\iu \lambda -A_{\Theta}$ is boundedly invertible for $\lambda \in \R\setminus \set{0}$, i.e., $\iu \R \setminus \set{0} \subseteq \uprho(A_{\Theta})$.
\end{lemma}

\begin{proof}
  Let $\lambda \in \R \setminus \set{0}$ and let us assume that $u = \begin{psmallmatrix}u_{1}\\u_{2}\end{psmallmatrix}$ is an eigenvector to the eigenvalue $\lambda$. Then from (\ref{eq:invSys}) we obtain that $u_{1},u_{2}\in \soboH^{1}(\Omega)$ and plugging the second equation into the first, we obtain that $u_{1}\in \soboH^{1}(\Omega)$ has to satisfy
  \begin{equation*}
    \lambda^{2}\rho u_{1} + \div T\nabla u_{1} = 0.
  \end{equation*}
  Combining the first equation of \eqref{eq:invSys} and \Cref{th:BdyOfEV} yields
  \begin{equation*}
    \boundtr u_{1} = 0 \quad\text{and}\quad \normaltr T\grad u_{1} = 0 \quad \text{on} \quad \supp k_{2}.
  \end{equation*}
  Given these boundary conditions, the unique continuation principle\footnote{It is here that our stronger assumptions from the beginning of the section come into play. They are not required anywhere else!}, cf.\ \Cref{sec:unique continuation}, implies that $u_{1}=0$ is the only solution, which in turn implies $u_{2}=0$ and therefore $u=0$, which is a contradiction.
\end{proof}

\begin{lemma}
  \label{th:zero-in-resolvent}
  $A_{\Theta}$ is boundedly invertible, i.e., $0 \in \uprho(A_{\Theta})$.
\end{lemma}
\begin{proof}
  Since $\dom A_{\Theta}\overset{\mathrm{cpt}}{\hookrightarrow} X$ by \Cref{th:compact-embedding}, it suffices to show that $0$ is not an eigenvalue.
  Suppose $0$ is an eigenvalue and $x \in \dom A_{\Theta}$ a corresponding eigenvector. Then from (\ref{eq:invSys}) we immediately see $x_{2} = 0$. Moreover, by \Cref{th:BdyOfEV}
  \begin{equation*}
    \normaltr T\grad x_{1} + k_{1}\boundtr x_{1} = 0.
  \end{equation*}
  We test the equation $0 = \div T \grad x_{1}$ with $x_{1}$ and integrate by parts:
  \begin{align*}
    0 =\scprod{\div T\grad x_{1}}{x_{1}}_{\Lp{2}(\Omega)}
      = -\scprod{T\grad x_{1}}{\grad x_{1}}_{\Lp{2}(\Omega)} + \dualprod{\normaltr T\grad x_{1}}{\boundtr x_{1}}_{\mp \frac{1}{2}}.
  \end{align*}
  Utilizing the boundary condition $\normaltr T\grad x_{1} + k\boundtr x_{1} = 0$ we obtain
  \begin{equation*}
    \scprod[\big]{\begin{psmallmatrix}x_{1} \\ x_{2}\end{psmallmatrix}}{\begin{psmallmatrix}x_{1} \\ x_{2}\end{psmallmatrix}}_{X}
    = \scprod[\big]{\tfrac{1}{\rho}x_{2}}{x_{2}}_{\Lp{2}(\Omega)} + \scprod{T \grad x_{1}}{\grad x_{1}}_{\Lp{2}(\Omega)} + \scprod{k_{1} \boundtr x_{1}}{\boundtr x_{1}}_{\Lp{2}(\tilde{\Gamma})}
    = 0\text{.}
  \end{equation*}
  Thus $x_{1} = 0$ as well, but $x = 0$ is not an eigenvector. Hence, $0$ cannot be an eigenvalue.
\end{proof}

Finally, we arrive at the stability result.

\begin{theorem}
  $A_{\Theta}$ generates a semi-uniformly stable semigroup.
\end{theorem}

\begin{proof}
  By \Cref{th:AThetaIsGen}, $A_{\Theta}$ generates a contraction semigroup. \Cref{onlyEV} shows that $\upsigma(A_{\Theta})$ consists only of eigenvalues, \Cref{th:imaginary-axis-in-resolvent} and \Cref{th:zero-in-resolvent} show that the imaginary axis is contained in the resolvent set of $A_{\Theta}$. \Cref{CritForSUS} then implies the claim.
\end{proof}

\begin{remark}
  We can also obtain semi-uniform stability of the full problem \eqref{FullWE} under suitable assumptions on the perturbation
  \begin{equation*}
    S\colon X\supseteq \dom (S)\to X,\quad
    \begin{pmatrix}x_{1} \\ x_{2}\end{pmatrix}
    \mapsto\begin{pmatrix} 0 \\ -ax_{1} - b \tfrac{1}{\rho}x_{2}\end{pmatrix}\text{.}
  \end{equation*}
  We see from \Cref{CritForSUS} that it suffices for at least suitably $A_{\Theta}$-bounded $S$ (to assure that $A_{\Theta} + S$ still generates a semigroup of contractions) that $S$ be negative.
\end{remark}

\section{Conclusion}

We established well-posedness of \Cref{FullWE} by constructing a boundary triple for the wave equation in the Lagrange representation. This framework allows us to formulate boundary conditions that involve the displacement, its velocity and its normal stress.
In particular, we showed that the proposed boundary conditions induce a maximally dissipative operator.
Well-posedness of \eqref{FullWE} is thus a simple consequence of the Lumer--Philips theorem (and simple perturbation theory).

For our results on stability we first point out, that the situation covered in this article does not grant access to the tools usually employed in the verification of stability of solutions, which in most instances means exponential stability. The simple reason is, that in the considered case here, the damping happens on the boundary. Because of the lack of stronger techniques, we cannot show exponential stability and an application of the Gearhart--Pr\"uss theorem seems out of reach, as we would need to prove sufficient resolvent estimates. The notion of ``semi-uniform stability'' presents a way out, with a convenient criterium (\Cref{CritForSUS}) requiring simple spectral theory. From the fact, that the domain of our differential operator $A_{\Theta}$ is compactly embedded into the state space (\Cref{th:compact-embedding}), we can immediately conclude that its spectrum is a pure point-spectrum, making the spectral condition of \Cref{CritForSUS} easy to check. We point out, that the bottleneck for the stability part is the application of a unique continuation theorem to prove that the invariant system (\ref{eq:invSys}) with zero boundary only admits the trivial solution.
Only in this instance do we require Lipschitz continuity of $T$ and $\rho$.


\appendix

\begingroup 
\crefalias{section}{appendix} 

\section{Friedrichs/Poincar{\'e} inequality}

We use a slightly modified version of Friedrichs/Poincar{\'e} inequality.
\begin{theorem}\label{th:Poincare}
  Let $\Omega$ be a bounded and connected Lipschitz domain, $\Gamma \subseteq \partial\Omega$ of positive measure and $k_{1} \in \Lb(\Lp{2}(\Gamma))$ such that $1 \notin \ker k_{1}$.
  Then there exists a $C > 0$ such that for all $f \in \soboH^{1}(\Omega)$:
  \begin{equation*}
    \norm{f}_{\Lp{2}(\Omega)} \leq C \bigl(\norm{\grad f}_{\Lp{2}(\Omega)} + \norm[\big]{k_{1}^{\frac{1}{2}}\boundtr f}_{\Lp{2}(\Gamma)}\bigr).
  \end{equation*}
\end{theorem}

\begin{proof}
  Assume that there is no such $C > 0$. Then we find a sequence $(f_{n})_{n\in\N}$ in $\soboH^{1}(\Omega)$ such that
  \begin{equation*}
    \norm{f_{n}}_{\Lp{2}(\Omega)} > n \bigl(\norm{\grad f_{n}}_{\Lp{2}(\Omega)} + \norm[\big]{k_{1}^{\frac{1}{2}} \boundtr f_{n}}_{\Lp{2}(\Gamma)}\bigr).
  \end{equation*}
  We define $g_{n} \coloneq \frac{f_{n}}{\norm{f_{n}}_{\Lp{2}(\Omega)}}$, which implies
  \begin{equation*}
    \norm{g_{n}}_{\Lp{2}(\Omega)} = 1,
    \quad
    \norm{\grad g_{n}}_{\Lp{2}(\Omega)} \to 0
    \quad\text{and}\quad
    \norm{k_{1}^{\frac{1}{2}}\boundtr g_{n}}_{\Lp{2}(\Gamma)} \to 0.
  \end{equation*}
  By the Rellich--Kondrachov theorem (\cite[Thm.~12.30]{Leoni2017}) and after passing to a subsequence, $(g_{n})_{n\in\N}$ converges to a $g\in\Lp{2}(\Omega)$ (w.r.t.\ $\norm{\argdot}_{\Lp{2}(\Omega)}$).
  For $\phi \in \Cc(\Omega)^{d}$ we have
  \begin{equation*}
    \scprod{\div \phi}{g}_{\Lp{2}(\Omega)}
    = \lim_{n\to\infty} \scprod{\div \phi}{g_{n}}_{\Lp{2}(\Omega)}
    = \lim_{n\to\infty} -\scprod{\phi}{\grad g_{n}}_{\Lp{2}(\Omega)}
    = 0,
  \end{equation*}
  which implies that $g \in \soboH^{1}(\Omega)$ and $\grad g = 0$. Consequently, $g$ is constant, i.e., there exists a $c \in \C$ such that $g = c$. Moreover, the sequence $(g_{n})_{n\in\N}$ converges also w.r.t.\ $\norm{\argdot}_{\soboH^{1}(\Omega)}$ to $g$. This shows
  \begin{equation*}
    k_{1} c = k_{1}\boundtr g  = k_{1}^{\frac{1}{2}}\lim_{n\to\infty} k_{1}^{\frac{1}{2}}\boundtr g_{n} = 0
  \end{equation*}
  and from the assumption on $k_{1}$ we deduce $c = 0$,
  which contradicts $\norm{g}_{\Lp{2}(\Omega)} = 1$.
\end{proof}

\section{Regularity}
\label{sec:regularity}

Note that for general (non-Lipschitz continuous) $T$, the $\conC^{\infty}$-functions are not necessarily dense in the domain of $\div T \grad$, as they may not even lie in the domain. Hence, arguments that rely on smooth functions in the domain of $\div T \grad$ and extend these properties by density, require an alternative dense set of regular functions.
It is possible to obtain (some of) the results of \Cref{sec:well-posedness} by identifying such a set of functions, but ultimately it is not required for our purposes. Nevertheless, we present a statement identifying a core for $A$ that can be viewed as a regularity statement for the boundary; specifically, about the existence of a meaningful $\Lp{2}(\partial\Omega)$ normal trace for $T \grad w$. We define the following spaces
\begin{align*}
  \cH(\div 0,\Omega) &\coloneq \dset{u \in \soboH(\div,\Omega)}{\div u = 0, \normaltr u = 0},\\
  \hH(\div,\Omega) &\coloneq \dset{u \in \soboH(\div,\Omega)}{\normaltr u \in \Lp{2}(\partial\Omega)}.
\end{align*}

We will disregard the additional boundary condition on $\Gamma_{0}$ imposed by $\cH_{\Gamma_{0}}^{1}(\Omega)$ (i.e., we regard $\Gamma_{0} = \emptyset$), as it actually simplifies the argument rather than complicating it, because in this space, $\norm{\grad \argdot}_{\Lp{2}(\Omega)}$ is equivalent to the full $\soboH^{1}(\Omega)$-norm.
The only non-obvious aspect in treating such a split boundary is to establish a Helmholtz decomposition, but fortunately this situation is covered by \cite[Thm.~5.3]{BaPaScho16}.

\begin{theorem}\label{th:regular-functions}
  The set $D_{1} \coloneq \dset{f \in \soboH^{1}(\Omega)}{T\grad f \in \hH(\div,\Omega)}$ is a core of $L = \div T \grad$ (with $\dom L = \dset{f \in \soboH^{1}(\Omega)}{T \grad f \in \soboH(\div,\Omega)}$ as an operator on $\Lp{2}(\Omega)$).
\end{theorem}

\begin{proof}
  Let $f \in \dom L$. Note that $\hH(\div,\Omega)$ is dense in $\soboH(\div,\Omega)$ w.r.t.\ $\norm{\argdot}_{\soboH(\div,\Omega)}$. Hence, for given $\epsilon > 0$ there exists a $g \in \hH(\div,\Omega)$ such that $\norm{T\grad f - g}_{\soboH(\div,\Omega)} \leq \epsilon$. Moreover, similar to \Cref{th:Helmholtz}, by replacing $\cgrad$, the gradient with Dirichlet boundary, with $\grad$, as in the proof there we can decompose $\Lp{2}(\Omega)$ into
  \begin{equation*}
    \Lp{2}(\Omega) = T\grad \soboH^{1}(\Omega) \oplus_{T^{-1}} \cH(\div 0,\Omega),
  \end{equation*}
  where $a \perp_{T^{-1}} b$ means $\scprod{T^{-1}a}{b}_{\Lp{2}(\Omega)} = 0$.
  Hence, there exist $h \in \soboH^{1}(\Omega)$ and $k \in \cH(\div 0,\Omega)$ such that $g = T \grad h + k$.
  Since $g,k \in \hH(\div,\Omega)$, we conclude that $T \grad h \in \hH(\div,\Omega)$. Moreover,
  \begin{equation*}
    \normaltr g = \normaltr T\grad h + \underbrace{\normaltr k}_{=\mathrlap{0}} = \normaltr T\grad h
  \end{equation*}
  and
  \begin{equation*}
    \div g = \div T \grad h + \underbrace{\div k}_{=\mathrlap{0}} = \div T \grad h.
  \end{equation*}
  This shows $\norm{\div T \grad (f - h)}_{\Lp{2}(\Omega)} = \norm{\div (T \grad f - g)}_{\Lp{2}(\Omega)} \leq \epsilon$.
  Clearly, we also have $\norm{T \grad f - g}_{\Lp{2}(\Omega)} \leq \epsilon$. Moreover, by orthogonality we obtain
  \begin{align*}
    \epsilon^{2}
    &\geq \norm{T \grad f - T \grad h - k}_{\Lp{2}(\Omega)}^{2}
    \geq \frac{1}{\norm{T^{-\frac{1}{2}}}^{2}} \norm{T \grad (f - h) - k}_{T^{-1}}^{2} \\
    &= \frac{1}{\norm{T^{-\frac{1}{2}}}^{2}} \bigl(\norm{T \grad(f-h)}_{T^{-1}}^{2} + \norm{k}_{T^{-1}}^{2}\bigr),
  \end{align*}
  which implies $\norm{T\grad(f-h)}_{T^{-1}} \leq \norm{T^{-\frac{1}{2}}}\epsilon$ and consequently $\norm{T\grad(f-h)}_{\Lp{2}(\Omega)} \leq \norm{T^{\frac{1}{2}}} \norm{T^{-\frac{1}{2}}} \epsilon$.
  Hence, we have
  \begin{equation*}
    \norm{T\grad(f - h)}_{\soboH(\div,\Omega)}
    \leq (1 + \norm{T^{\frac{1}{2}}}\norm{T^{-\frac{1}{2}}}) \epsilon.
  \end{equation*}
  Moreover, identifying $\operatorname{ker}\nabla = \C$, we have the decomposition
  \begin{equation*}
    \soboH^{1}(\Omega) = \C^{\perp} \oplus \C
  \end{equation*}
  and that $\grad \colon \C^{\perp} \to \ran \grad$ is a boundedly invertible mapping, where $\ran \grad$ is endowed with $\norm{\argdot}_{\Lp{2}(\Omega)}$. Note that $f$ and $h$ can be decomposed according to this decomposition into $f = f_{1} + f_{2}$ and $h = h_{1} + h_{2}$, where $f_{1},h_{1} \in \C^{\perp}$ and $f_{2},h_{2} \in \C$. Since $T$ is boundedly invertible, we obtain
  \begin{align*}
    \norm{f_{1} - h_{1}}_{\Lp{2}(\Omega)}
    = \norm{\grad^{-1}T^{-1}T\grad(f - h)}_{\Lp{2}(\Omega)}
    \leq C \norm{T^{\frac{1}{2}}}\norm{T^{-\frac{1}{2}}} \epsilon.
  \end{align*}
  We define $\phi = h_{1} + f_{2} \in \soboH^{1}(\Omega)$ and obtain $\grad \phi = \grad h$; in particular $\phi \in D_{1}$. This gives
  \begin{align*}
    \norm{f - \phi}_{\dom L}
    &= \sqrt{\norm{f - \phi}_{\Lp{2}(\Omega)}^{2} + \norm{\div T \grad (f - \phi)}_{\Lp{2}(\Omega)}^{2}} \\
    &= \sqrt{\norm{f_{1} - h_{1}}_{\Lp{2}(\Omega)}^{2} + \norm{\div T \grad (f - h)}_{\Lp{2}(\Omega)}^{2}}
    \leq \tilde{C} \epsilon
  \end{align*}
  for a constant $\tilde{C} > 0$. Note that by construction, $\phi \in D_{1}$ and $\norm{f - \phi}_{\dom L} \leq \hat{C} \epsilon$, which proves that $D_{1}$ is a core of $L$.
\end{proof}

Consequently, for $A = \begin{psmallmatrix}0 & \tfrac{1}{\rho} \\ \div  T \grad & 0 \end{psmallmatrix}$ with $\Gamma_{0} = \emptyset$ from \Cref{sec:well-posedness}, we obtain:
\begin{corollary}
  The set
  \begin{equation*}
    D\coloneq \dset*{\begin{psmallmatrix} x_{1} \\ x_{2} \end{psmallmatrix} \in \dom(A)}{\normaltr T\grad x_{1}\in \Lp{2}(\partial \Omega)} = D_{1} \times \rho \soboH^{1}(\Omega)
  \end{equation*}
  is a core of $A$.
\end{corollary}

\begin{remark}
  Note that the general statement with non-empty $\Gamma_{0}$ can be obtained relatively similar. The only difference is that in the proof of \Cref{th:regular-functions} we need the corresponding Helmholtz decomposition
  \begin{equation*}
    \Lp{2}(\Omega) = T \grad \cH^{1}_{\Gamma_{0}}(\Omega) \oplus_{T^{-1}} \cH_{\partial\Omega \setminus \Gamma_{0}}(\div 0,\Omega),
  \end{equation*}
  see, e.g., \cite[Thm.~5.3]{BaPaScho16}.
\end{remark}

\section{Unique continuation}%
\label{sec:unique continuation}

In this section we want to briefly explain, how we apply the unique continuation principle. In particular, we use the version from \cite{GarofaloLin1987}. To compare the notation, we have $A = T$, $\mathbf{b}=0$ and $V=\lambda^{2} \rho$ (on the left-hand side of the equality sign is the notation from \cite{GarofaloLin1987}). We impose the same assumption on $T$ and $\rho$ as in \Cref{sec:stability}, i.e., $T$ and $\rho$ are Lipschitz-continuous.

Let $\Omega \subseteq \R^{d}$ be a bounded and connected Lipschitz domain, $\Gamma \subseteq \partial\Omega$ open (in $\partial\Omega$) and $w \in \soboH^{1}(\Omega)$ be a (weak) solution of
\begin{align}\label{eq:eigenvalue-problem}
  \begin{aligned}
    \div T \grad w + \lambda^{2} \rho w &= 0 \quad \text{in}\ \Omega, \\
    \nu \cdot T \grad w &= 0 \quad \text{on}\ \Gamma, \\
    w &= 0 \quad \text{on}\ \Gamma.
  \end{aligned}
\end{align}
We define for every $x \in \Omega$ the set $\Omega_{x}$ as in \Cref{fig:unique-continuation-Omega-x}, i.e., we choose a smooth path from $x$ to a point outside of $\Omega$ that crosses $\partial\Omega$ in (the interior of) $\Gamma$ and we define $\Omega_{x}$ as a smooth neighborhood of this path that is sufficiently small (it does not intersect $\partial\Omega$ outside of $\Gamma$). Furthermore, we define
\begin{equation}
  \label{eq:wx-def}
  w_{x}(\zeta) \coloneq
  \begin{cases}
    w(\zeta), & \zeta \in \Omega \cap \Omega_{x}, \\
    0, & \zeta \in \Omega_{x} \setminus \Omega.
  \end{cases}
\end{equation}%
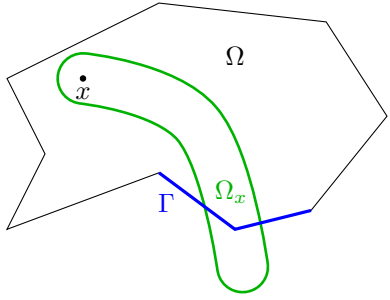
\begin{figure}[ht]%
  \centering
  \begin{tikzpicture}
    \coordinate (p1) at (2,0.75);
    \coordinate (p2) at (3,0);
    \coordinate (p3) at (4,0.25);

    \coordinate (x) at (1,2);

    \draw[line width=20pt, line cap=round, draw=green!70!black] plot [smooth, tension=0.7] coordinates {(x) (2.5,1.4)  ($(p2) + (0.1,-0.5)$)};
    \draw[line width=18pt, line cap=round, draw=white] plot [smooth,tension=0.7] coordinates {(x) (2.5,1.4)  ($(p2) + (0.1,-0.5)$)};

    \draw[fill=black] (x) circle (1pt);
    \node[below] at (x) {$x$};

    \coordinate (x-above) at ($(x) + (0,0.3)$);
    \coordinate (x-left) at ($(x) + (-0.4,0)$);
    \coordinate (x-below) at ($(x) + (0,-0.4)$);

    \draw (0,0) -- (0.5,1) -- (0,2) -- (2,3) -- (4.2,2.75) -- (5,1.5) -- (p3) -- (p2) -- (p1) -- cycle;

    \draw[very thick,blue] (p1) -- (p2) -- (p3);




    \node[xshift=-0.05cm,yshift=0.5cm,color=green!70!black] at (p2) {$\Omega_{x}$};
    \node at (3,2.3) {$\Omega$};
    \node[yshift=-0.4cm,xshift=0.1cm,blue] at (p1) {$\Gamma$};

  \end{tikzpicture}
  \caption{\label{fig:unique-continuation-Omega-x}Deriving boundary unique continuation from strong unique continuation.}
\end{figure}%
Note that we can extend $T$ and $\rho$ Lipschitz-continuously and boundedly to all of $\R^{d}$.\footnote{We can Lipschitz-continuously extend $T$ to a neigborhood of $\Omega$ and we then work with $\alpha T + (1-\alpha) \idop$, where $\alpha$ is a cutoff function. This construction gives an extension that is globally bounded. Analogously, we can extend $\rho$.}
By splitting the area of integration, using integration by parts and the boundary conditions on $\Gamma$, we can see that $w_{x} \in \soboH^{1}(\Omega_{x})$ such that $T\grad w_{x} \in \soboH(\div,\Omega_{x})$. In particular, this leads to $w_{x}$ solving (shown in \Cref{th:wx-solves-ev-problem})
\begin{equation*}
  \div T \grad w + \lambda^{2} \rho w = 0 \quad \text{in}\ \Omega_{x}.
\end{equation*}
We aim to apply the unique continuation principle \cite[Thm.~1.1]{GarofaloLin1987} to conclude that $w_{x}$ is zero. The assumptions of \cite[Thm.~1.1]{GarofaloLin1987} are easy to check for our setting, since $\mathbf{b} = 0$ and $V = \lambda^{2}\rho$ is bounded.
Hence, $w_{x} = 0$ for every $x \in \Omega$ and consequently also $w=0$.

\begin{lemma}%
  \label{th:wx-solves-ev-problem}
  Let $w \in \cH^{1}_{\Gamma}(\Omega)$ be a weak solution of \eqref{eq:eigenvalue-problem} and $w_{x}$ be defined as in \eqref{eq:wx-def}. Then $w_{x} \in \soboH^{1}(\Omega_{x})$ and $T\grad w_{x} \in \soboH(\div,\Omega_{x})$ with
  \begin{equation*}
    \grad w_{x} =
    \begin{cases}
      \grad w(\zeta), & \zeta \in \Omega \cap \Omega_{x}, \\
      0, & \zeta \in \Omega_{x} \setminus \Omega,
    \end{cases}
    \quad\text{and}\quad
    \div T\grad w_{x} =
    \begin{cases}
      \div T \grad w(\zeta), & \zeta \in \Omega \cap \Omega_{x}, \\
      0, & \zeta \in \Omega_{x} \setminus \Omega.
    \end{cases}
  \end{equation*}
  Moreover, $w_{x}$ satisfies $\div T \grad w_{x} + \lambda^{2} \rho w_{x} = 0$ in $\Omega_{x}$.
\end{lemma}

\begin{proof}
  Let $\phi \in \Cc(\Omega_{x};\C^{d})$. Then $\phi \in \Cc(\R^{d};\C^{d})$ such that $\supp \phi \cap (\partial\Omega \setminus \Gamma) = \emptyset$.
  This allows us to use integration by parts on $\Omega$ for $w \in \cH^{1}_{\Gamma}(\Omega)$ and $\phi$ without boundary terms (weak characterization of $\cH^{1}_{\Gamma}(\Omega)$, cf.~\cite{BaPaScho16}).
  Moreover, $\phi$---and therefore also $\div \phi$---is $0$ on $\Omega \setminus \Omega_{x}$.  
  In the following, we will only specify the domain of the $\Lp{2}$ space in the index of the inner product, e.g., $\scprod{\argdot}{\argdot}_{\Omega}$ instead of $\scprod{\argdot}{\argdot}_{\Lp{2}(\Omega)}$.
  \begin{align*}
    \scprod{w_{x}}{\div \phi}_{\Omega_{x}}
    &= \scprod{\underbrace{w_{x}}_{=\mathrlap{w}}}{\div \phi}_{\Omega_{x} \cap \Omega}
    + \scprod{\underbrace{w_{x}}_{=\mathrlap{0}}}{\div \phi}_{\Omega_{x} \setminus \Omega}
    = \scprod{w}{\div \phi}_{\Omega} + \scprod{w}{\underbrace{\div \phi}_{=\mathrlap{0}}}_{\Omega \setminus \Omega_{x}} \\
    &\stackrel{\mathllap{\text{Int.\ by parts}}}{=} -\scprod{\grad w}{\phi}_{\Omega}
      = -\scprod{\grad w}{\phi}_{\Omega_{x} \cap \Omega} - \scprod{0}{\phi}_{\Omega_{x} \setminus \Omega}. 
  \end{align*}
  Hence, $w_{x} \in \soboH^{1}(\Omega_{x})$ satisfies
  \begin{equation*}
    \grad w_{x} =
    \begin{cases}
      \grad w(\zeta), & \zeta \in \Omega \cap \Omega_{x}, \\
      0, & \zeta \in \Omega_{x} \setminus \Omega.
    \end{cases}
  \end{equation*}
  We proceed analogously to show that $T\grad w_{x} \in \soboH(\div,\Omega)$:
  We note that since $w$ is a weak solution, we automatically get $T \grad w \in \cH_{\Gamma}(\div,\Omega)$.\footnote{Roughly speaking, $\cH_{\Gamma}(\div,\Omega)$ is $\dset{f \in \soboH(\div,\Omega)}{\normaltr f \big\vert_{\Gamma} = 0}$. The problem is, that the restriction to $\Gamma$ in general is not defined. However, there are two approaches that make this statement rigorous (a ``strong'' one via completion of regular functions and a ``weak'' one via inner products and test functions). Both approaches fortunately coincide, cf.~\cite{BaPaScho16}. There, the authors use the notation $\mathring{\mathsf{D}}_{\Gamma}(\Omega)$ (and $\mathring{\mathcal{D}}_{\Gamma}(\Omega)$) for this set.}
  Let $\psi \in \Cc(\Omega_{x};\C)$. Then $\psi \in \Cc(\R^{d};\C)$ such that $\supp \psi \cap (\partial\Omega \setminus \Gamma) = \emptyset$. Again, this allows us to use integration by parts on $\Omega$ without boundary terms (weak characterization of $\cH_{\Gamma}(\div,\Omega)$, cf.\ \cite{BaPaScho16}).
  Hence,
  \begin{multline*}
    \scprod{T \grad w_{x}}{\grad \psi}_{\Omega_{x}}
    = \scprod{T \grad w}{\grad \psi}_{\Omega} = - \scprod{\div T \grad w}{\psi}_{\Omega} \\
    = - \scprod{\div T \grad w}{\psi}_{\Omega_{x} \cap \Omega} - \scprod{0}{\psi}_{\Omega_{x} \setminus \Omega},
  \end{multline*}
  which shows $T \grad w_{x} \in \soboH(\div,\Omega_{x})$ and
  \begin{equation*}
    \div T\grad w_{x} =
    \begin{cases}
      \div T \grad w(\zeta), & \zeta \in \Omega \cap \Omega_{x}, \\
      0, & \zeta \in \Omega_{x} \setminus \Omega.
    \end{cases}
  \end{equation*}
  Since $\div T \grad w + \lambda^{2} \rho w = 0$ holds, we immediately infer the same identity for $w_{x}$.
\end{proof}

\endgroup 

\end{document}